\documentclass[10pt]{amsart}
\usepackage[margin=1.2in,marginparsep=0.1in,marginparwidth=1in]{geometry}
\usepackage{amssymb,amsmath,amsthm,amstext,amscd,latexsym,graphics,graphicx,bbm,caption}
\usepackage[usenames,dvipsnames,svgnames,table]{xcolor}
\usepackage[plainpages=false,colorlinks=true, pagebackref]{hyperref}
\usepackage{tikz}
\usetikzlibrary{positioning}

\tikzset{>=stealth}
\hypersetup{citecolor=Sepia,linkcolor=blue, urlcolor=blue}

\usepackage{cite}   

\makeatletter
\def\@tocline#1#2#3#4#5#6#7{\relax
  \ifnum #1>\c@tocdepth 
  \else
    \par \addpenalty\@secpenalty\addvspace{#2}%
    \begingroup \hyphenpenalty\@M
    \@ifempty{#4}{%
      \@tempdima\csname r@tocindent\number#1\endcsname\relax
    }{%
      \@tempdima#4\relax
    }%
    \parindent\z@ \leftskip#3\relax \advance\leftskip\@tempdima\relax
    \rightskip\@pnumwidth plus4em \parfillskip-\@pnumwidth
    #5\leavevmode\hskip-\@tempdima
      \ifcase #1
       \or\or \hskip 2em \or \hskip 2em \else \hskip 3em \fi%
      #6\nobreak\relax
    \dotfill\hbox to\@pnumwidth{\@tocpagenum{#7}}\par
    \nobreak
    \endgroup
  \fi}
\makeatother

\newtheorem{intro-thm}{Theorem}[]
\theoremstyle{plain}
\newtheorem{thm}{Theorem}[section]
\newtheorem{theorem}[thm]{Theorem}

\newtheorem{lemma}[thm]{Lemma}

\theoremstyle{definition}
\newtheorem{remark}[thm]{Remark}

\newtheorem{notation}[thm]{Notation}

\newtheorem{definition}[thm]{Definition}

\newcommand{\rdown}[1]{\lfloor{#1}\rfloor}

\newcommand{\inj}{\hookrightarrow}
\newcommand{\tensor}{\otimes}

\newcommand{\Union}{\bigcup}
\newcommand{\intersection}{\cap}
\newcommand{\union}{\cup}

\newcommand{\del}{\partial}

\newcommand{\Ker}{{\rm Ker  }}

\newcommand{\Spec}{{\rm Spec \,}}

\newcommand{\Char}{{\rm char}}

\renewcommand{\tilde}{\widetilde}

\newcommand{\sC}{{\mathcal C}}

\newcommand{\sO}{{\mathcal O}}
\newcommand{\sP}{{\mathcal P}}
\newcommand{\sQ}{{\mathcal Q}}

\newcommand{\sT}{{\mathcal T}}

\newcommand{\A}{{\mathbb A}}

\newcommand{\N}{{\mathbb N}}

\let\syn\mathsf
\newcommand{\Step}[1]{\underline{\syn{Step \ {#1}}}}
\newcommand{\scr}{\scriptscriptstyle}

\newcommand{\dg}{\syn{deg}}

\newcommand{\hash}{\raisebox{.3ex}{\footnotesize \# }}

\input{xy}
\xyoption{all}

\begin{document}

\title{G\MakeLowercase{abber's presentation lemma for finite fields}}
\author{{A\MakeLowercase{mit} H\MakeLowercase{ogadi}} \MakeLowercase{and}
{G\MakeLowercase{irish} K\MakeLowercase{ulkarni}}}



\date{}

\begin{abstract}
We give a proof of Gabber's presentation lemma for finite fields. We first prove this lemma in the special case of open subsets of the affine plane using ideas from Poonen's proof of Bertini's theorem over finite fields. We then reduce the case of general smooth varieties to this special case. 
\end{abstract}

\maketitle

\section{Introduction}

A presentation lemma proved by Gabber in \cite{gabber} (see also \cite{chk}) plays a foundational role in $\A^1$-algebraic topology as developed by F. Morel in  \cite{Morel}. This lemma can be thought of as an algebro-geometric analogue of tubular neighbourhood theorem in differential geometry.  The current published proof of Gabber's presentation lemma works only over infinite fields. In a private communication to F. Morel, Gabber has pointed out that the proof of this theorem also holds for finite fields. Unfortunately there is no published proof for this case.

The goal of this paper is to prove the following version of Gabber's presentation lemma over finite fields. 

\begin{theorem}\label{gabberfinite}
Let $X$ be a smooth variety of dimension $d\geq 1$ over a finite field $F$ and $Z\subset X$ be a closed subvariety. Let $p\in Z$ be a point. Let $\A^d_F \xrightarrow{\pi} \A^{d-1}_F$ denote the projection onto the first $d-1$ coordinates.  Then there exists
\begin{enumerate}
\item[(i)] an open neighbourhood $ U\subset X$  of $p$,
\item[(ii)]  a map $\Phi:U\to \A_F^d$,
\item[(iii)] an open neighbourhood $V\subset \A^{d-1}_F$ of $\Psi(p)$ where 
 $\Psi:U\to \A_F^{d-1}$ denotes the composition
$$ U\xrightarrow{\Phi}\A_F^d \xrightarrow{\pi} \A_F^{d-1}$$
\end{enumerate}
such that 

\begin{enumerate}
\item $\Phi$ is \'{e}tale. 
\item $\Psi_{|Z_V}:Z_V \to V$ is finite where $Z_V:=Z\intersection \Psi^{-1}(V)$. 
\item $\Phi_{|Z_V}:Z_V \to \A^1_V= \pi^{-1}(V)$ is a closed immersion. 
\end{enumerate}
\end{theorem}
\begin{remark}\label{simplify}
Without loss of generality, we may (and will) assume henceforth that $X$ is affine. Moreover, by \cite[3.2]{chk}, we may also assume that  $Z$ is a principal divisor defined by $f\in \sO(X)$ and $p$ is a closed point. 
\end{remark}

\begin{remark}
If one finds $U,\Phi,V,\Psi$ satisfying conditions (1),(2),(3) of  \eqref{gabberfinite}, one can also arrange, by shrinking $U$ if necessary, the following additional condition 
\begin{enumerate}
\item[(4)]  $Z_V=\Phi^{-1}\Phi(Z_V)$.
\end{enumerate} 
To see this, let $\widetilde{Z}$ denote the image of $Z_V$ in $\A^1_V$. The morphism $\Phi^{-1}\Phi(Z_V)\to \widetilde{Z}$ is \'{e}tale, and it admits a section as $Z_V$ maps isomorphically onto $\widetilde{Z}$. Thus $\Phi^{-1}\Phi(Z_V)$ is a disjoint union of $Z_V$ and a closed subset  $T$ of $\Psi^{-1}(V)$. Replacing $U$ by $U\backslash T$, one sees that the additional condition $(4)$ is satisfied.  
\end{remark}

The proof of Gabber's presentation lemma for infinite fields (see  \cite[3.1]{gabber} or \cite[3.1]{chk}) shows that for the maps $\Phi,\Psi$ appearing in the statement of the lemma, suitable generic choices work.  The problem in making this proof work over a finite field is very similar to the problem of making Bertini's theorem work over a finite field. Bertini's theorem for finite fields  was proved by Poonen in \cite{poonen} using an extremely clever counting argument. Because of the broad similarities of the issues involved, it is natural to try to use Poonen's  argument to prove Gabber's presentation lemma over finite fields. However, Poonen's counting argument, in our opinion, is easier to apply in the case of subvarieties of an affine space. Thus, the first  step of the proof of Theorem \ref{gabberfinite} is reduction to the case where $X$ is an open subset of $\A^d_F$. This is done in section \ref{firstred}. Unfortunately, we found even this case too complex to directly apply Poonen's ideas from \cite{poonen}. Fortunately, we are able to reduce this complexity by using induction on $d$ to reduce to the case where $d=2$, i.e. $X$ is an open subset of $\A^2_F$. This is done in  section \ref{inductiond}. This induction argument, although short, was one of the most time-taking tasks for us in proving Theorem \ref{gabberfinite}.   An important ingredient of this induction is a slightly modified version of Noether normalization trick (see \ref{nntrick}).  

The case of open subsets of $\A^2_F$ is now ideal for using Poonen's counting argument. Indeed, the handling of points of small degree is very similar to that of \cite{poonen}. However, we are unable to handle the error term for `high degree points' as is done in \cite{poonen}. We fix this with a small trick (\ref{hexists}). \\

\noindent {\bf Acknowledgements}: We thank F. Morel for his comments and for answering our questions on the current status of this result.  We thank A. Asok, F. D\'eglise, M. Levine and J. Riou for their comments during the early stage of this project. We thank Anand Sawant and Charanya Ravi for pointing out a mistake in the previous version of the paper. We also thank the referee for numerous suggestions.

\section{Reduction to open subsets of $\A^d_F$}\label{firstred}

The goal of this section is to prove Lemma \ref{redopen}, which reduces Theorem \ref{gabberfinite} to the case where $X$ is an open subset of $\A^d_F$ and $p\in Z\subset X$ is a closed point with first $d-1$ coordinates equal to $0$.

\begin{notation}\label{notation} Throughout this paper we work over a fixed finite field $F$. We further fix the following notation. 
\begin{enumerate}
\item Let $Y$ be a subset of a scheme $X/F$. We let $Y_{ \scr \leq r}:= \{x\in Y\ |\ \dg(x) \leq r\}$ and similarly $Y_{ \scr < r}:= \{x\in Y\ |\ \dg(x) < r\}$ and $Y_{ \scr = r}:= \{x\in Y\ |\ \dg(x) = r\}$.  
\item For $f_1,...,f_i\in F[X_1,...,X_n]$ we let $\syn{Z}(f_1,...,f_i)$ denote the closed subscheme of $\A^n_F$ defined by the ideal  $(f_1,...,f_i)$. 
\end{enumerate}
\end{notation}

We first start by recalling the following standard trick (see \cite{mum}) used in the proof of Noether's normalization lemma. 

\begin{lemma}\cite[page 2]{mum}\label{nntrick1}
 Let $k$ be any field and $n\geq 1$ be any integer. Let $Z/k$ be a finitely generated affine scheme of dimension at most  $n-1$.  Let $$Z\xrightarrow{(\phi_1,...,\phi_n)}\A^{n}_k$$ be a finite map. Let  $Q (T)\in k[T]$ be a non constant monic polynomial and $Q=Q(\phi_n)$. Then for $\ell>>0$, the map 
$$ Z\xrightarrow{(\phi_{\scr 1}-Q^{\ell^{n-1}},\ldots,\phi_{\scr n-1}-Q^{\ell})} \A^{n-1}_k$$ is finite.
\end{lemma}

\begin{remark}\label{q=phin}
We claim that finiteness of $Z\xrightarrow{(\phi_1,...,\phi_n)}\A^{n}_k$ implies that of $ Z \xrightarrow{(\phi_1,...,Q)}\A^{n}_k$. This is because the later map is a composition of the following two finite maps
 $$Z\xrightarrow{(\phi_1,...,\phi_n)}\A^{n}_k \xrightarrow{(Y_1,\ldots ,Q(Y_n))}\A^{n}_k.$$
 One can thus easily reduce the proof of the above general case to the case where $Q(T)=T$.  Unless explicitly mentioned, we will usually assume $Q(T)=T$ while applying the lemma. As in the proof of Noether normalization, the above lemma is usually applied repeatedly until one gets a map from $Z$ to $\A^{\dim(Z)}_k$. 
\end{remark}

\begin{lemma}\label{redopen} 
Let $p\in Z\subset X$ be as in Theorem \ref{gabberfinite}. Further, assume that $X$ is affine, $Z$ is a principal divisor and $p$ is a closed point (see Remark \ref{simplify}). Then there exists a map $\varphi:X\to \A^d_F$ and an open neighbourhood $W$ of $\varphi(p)$ such that 
\begin{enumerate}
\item $\varphi^{-1}(W) \to W$ is \'{e}tale. 
\item $Z_W:= Z\intersection \varphi^{-1}(W) \to W$ is a closed immersion. 
\item The first $d-1$ coordinates of $\varphi(p)$ are $0$. 
\end{enumerate}
In particular, it suffices to prove Theorem \ref{gabberfinite} where $X$ is an open subset of $\A^d_F$ and the first $d-1$ coordinates of $p$ are zero. 
\end{lemma}
\begin{proof}   Let
\begin{enumerate}
\item[-] $X=\Spec(A)$.
\item[-] $Z= \Spec(A/(f))$ and let $\overline{A}:= A/(f)$. 
\item[-] $\mathfrak{m}\subset A$ be the maximal ideal of the closed point $p$.
\item[-] $F(p)$ denote the residue field of $p$.

\end{enumerate}

\noindent $\Step{1}$:
Since $X/F$ is smooth, $\dim_{F(p)}(\mathfrak{m}/\mathfrak{m}^2)=d$. Choose $\{x_{\scr 1},...,x_{\scr d-1}\} \subset {\mathfrak m}$ such that they span a $d-1$ dimensional $F(p)$-subspace of $\mathfrak{m}/\mathfrak{m}^2$. In this step we claim that  there exists $y \in A$ such that 
\begin{enumerate}

\item $y\   {\rm mod}\  \mathfrak{m}$  is a primitive element of $ F(p)/F$.
\item The set  $\{x_1,...,x_{d-1},h(y)\}$ (modulo $\mathfrak{m}^2$) gives a $F(p)$-basis  of ${\mathfrak m}/{\mathfrak m}^2$, where $h$ is the minimal polynomial of $y\   {\rm mod}\  \mathfrak{m}$.
\item  The map $(x_1,\ldots,x_{d-1},{y}):X \xrightarrow{\eta} \A^d$ is \'{e}tale at $p$.
\item The map $\eta$ induces an isomorphism on residue fields $F(\eta(p)) \to F(p)$. 
\end{enumerate}

Now let $w\in \mathfrak{m}$ be an element such that $\{x_{\scr 1},...,x_{\scr d-1},w\}$ span $\mathfrak{m}/\mathfrak{m}^2$ as a $F(p)$-vector space.  Let $c$ be a primitive element of $F(p)/F$ and $h$ be its minimal polynomial. Choose $\hat{y}\in A$ such that
$$\hat{y}  \equiv c \mod \mathfrak{m}.$$ Since $c$ is separable over $F$, $h'(c)\neq 0$. Thus $h'(\hat{y}) \notin \mathfrak{m}$ or equivalently $h'(\hat{y})$ is a unit in the ring $A/\mathfrak{m}^2$.
Choose $\epsilon\in \mathfrak{m}$  such that 
$$ \epsilon \equiv \frac{w-h(\hat{y})}{h'(\hat{y})} \ \syn{mod} \ \mathfrak{m}^2.$$
Thus the $F(p)$-span of $\{x_{\scr 1},...,x_{\scr d-1},h(\hat{y})+\epsilon h'(\hat{y})\}$ is   $\mathfrak{m}/\mathfrak{m}^2$.
Let
$${y}=\hat{y}+\epsilon.$$
We note that $$h({y})=h(\hat{y}+\epsilon) \equiv h(\hat{y})+\epsilon h'(\hat{y}) \mod 
\mathfrak{m}^2.$$
Hence $\{x_1,\ldots,x_{d-1},h(y)\}$ gives a $F(p)$-basis for $\mathfrak{m}/\mathfrak{m}^2$.

Now let $\eta$ be the map $(x_1,\ldots,x_{d-1},{y}):X \xrightarrow{} \A^d_F$. Since $y \ \syn{mod} \ \mathfrak{m}$ is a primitive element of $F(p)$, one observes that $F(\eta(p)) \to F(p)$ is an isomorphism. It remains to show that $\eta$ is \'{e}tale at $p$.  The maximal ideal of $\eta(p)$ in $F[X_1,...,X_d]$ is $\mathfrak{n}=(X_1,\ldots,X_{d-1},h(X_d))$.  
As $\{x_1,\ldots,x_{d-1},h(y)\}$ is a $F(p)$-basis for $\mathfrak{m}/\mathfrak{m}^2$, {that $\eta$ is \'{e}tale at $p$} follows from the surjectivity of 
$$\mathfrak{n}/\mathfrak{n}^2\xrightarrow{\eta^*}\mathfrak{m}/\mathfrak{m}^2 .$$

\noindent$\Step {2}$:
Let $U$ be an open neighbourhood  of $p$ in $X$ such that $\eta_{\scr{|U}}$ is \'{e}tale. Let $$ B= (X\setminus U) \ \sqcup Z.$$
In this step we modify $x_1,\ldots,x_{d-1}$ to $z_1,\ldots,z_{d-1}$ so that
\begin{enumerate}
\item  The map $\tilde{\eta}=(z_1,\ldots,z_{d-1},{y}):X \rightarrow \A^d_F$ is \'{e}tale on $U$.
\item The set $\{z_1,\ldots,z_{d-1},h({y})\}$ is a $F(p)$ basis for $\mathfrak{m}/\mathfrak{m}^2$. 
\item The map $B \xrightarrow{(z_1,\ldots,z_{d-1})}\A^{d-1}_F$ is finite.
\end{enumerate}

Let $\tilde{A}:= A/I(B)$ and $\tilde{\mathfrak{m}}$ denote the image of $\mathfrak{m}$ in $\tilde{A}$. For any element $\alpha \in A$, let $\tilde{\alpha}$ denote its image in $\tilde{A}$. Choose $y_1,...,y_m \in A$  which generate $A$ as an $F$ algebra. We expand this generating set to include the $x_i$'s. In particular  
\begin{align*}
A & = F[x_{\scr 1},...,x_{\scr d-1},y_{\scr 1},...,y_{\scr m}],\\
\tilde{A} & = F[\tilde{x}_{\scr 1},...,\tilde{x}_{\scr d-1},\tilde{y}_{\scr 1},...,\tilde{y}_{\scr m}].
\end{align*}
The image of $y_{\scr i}$ in $A/\mathfrak{m}$ satisfies a non-constant monic polynomial, say $f_{\scr i}$, over $F$. Let 
\begin{align*}
y_{\scr i,0} & \ :=\ f_i(y_{\scr i})  \in \mathfrak{m}.\\
x_{\scr i,0} & \ := \ x_{\scr i} \\ 
A_{\scr 0} & \ := \ F[x_{\scr 1,0},..,x_{\scr d-1,0},y_{\scr 1,0},...,y_{\scr m,0}] \\
\tilde{A}_{\scr 0} & \ := \ F[\tilde{x}_{\scr 1,0},..,\tilde{x}_{\scr d-1,0},\tilde{y}_{\scr 1,0},...,\tilde{y}_{\scr m,0}]
\end{align*}
Clearly, $\tilde{A}$ is finite over $\tilde{A}_{\scr 0}$. \\

  For $0 \leq r \leq m-1$, we inductively define $A_{\scr r+1}$ and elements $x_{\scr i,r+1}, y_{\scr i, r+1}$ as follows. By \ref{nntrick1}, we choose an integer $\ell_{\scr r}>1$ such that the following definitions make $\tilde{A}_{\scr r}$ a finite $\tilde{A}_{\scr r+1}$-algebra. Since any sufficiently large choice of $\ell_r$ works, we assume that $\ell_r$ is a multiple of the $\Char(F)$. Let
\begin{align*}
 x_{\scr i,r+1}:= & \ \  x_{\scr i,r}-(y_{\scr m-r,r})^{\ell_{\scr r}^i} & \ \forall \ 1\leq i \leq d-1\\
y_{\scr i,r+1}:= & \ \  y_{\scr i,r}-(y_{\scr m-r,r})^{\ell_{\scr r}^{d-1+i}} & \   \forall \ 1\leq i \leq m-r-1 \\
A_{\scr r+1} := & \ \   F[x_{\scr 1,r+1},..,x_{\scr d-1,r+1},y_{\scr 1,r+1},...,y_{\scr m-r-1,r+1}] \\
\tilde{A}_{\scr r+1} := & \ \   F[\tilde{x}_{\scr 1,r+1},..,\tilde{x}_{\scr d-1,r+1},\tilde{y}_{\scr 1,r+1},...,\tilde{y}_{\scr m-r-1,r+1}]
 \end{align*}
Since, $x_{\scr i,0}$ and $y_{\scr i,0}$ belong to $ \mathfrak{m}$, inductively one can observe 
\begin{align*}
& y_{\scr i,r} \in \mathfrak{m} \\
& x_{\scr i,r} \in \mathfrak{m} \\
& x_{\scr i,r+1} \equiv  x_{\scr i,r} \ {\rm mod} \ \mathfrak{m}^2  
\end{align*}
For ease of notation, let us rename 
$$ z_{\scr i}:= x_{\scr i,m}.$$
Note that for all $i\leq d-1$, $ z_i- x_i$ is of the form $\beta_i^{k_i}$ for  $\beta_i\in \mathfrak{m}$ and an integer $k_i$ divisible by $\Char(F)$. This ensures requirements (1) and (2) of Step 2.
Recall that $m$ is an integer such that $\{y_1,...,y_m\}$ are the chosen generators of $A$ as an $F$ algebra. 
It is now straightforward to see that $\{z_{\scr 1},...,z_{\scr d-1}\}\subset \mathfrak{m}$ such that  $\tilde{A}$ is a finite algebra over $\tilde{A}_{\scr m}=F[\tilde{z}_{\scr 1},...,\tilde{z}_{\scr d-1}]$.  \\

\noindent $\Step{3}$: In this step we will further modify $y$ while ensuring that $(1)$ and $(2)$ of the above step continue to hold. Since the map $\tilde{\eta}_{\scr |B}:B\to \A^{d-1}_F$ is finite, there exists finitely many points $\{p,p_1,...,p_t\}\subset B$ which are contained in the zero locus $\syn{Z}(z_{\scr 1},...,z_{\scr d-1})$. Let $\mathfrak{m}_i$ be the maximal ideal corresponding to $p_i$ for $1\leq i \leq t$. 
By Chinese remainder theorem, choose $\delta \in A$ such that
\begin{align*}
\delta \equiv 0 & \ \syn{mod}\ \mathfrak{m} & \\ 
\delta^{\scr {\rm char}(F)} \equiv {-y} & \ \syn{mod} \ \mathfrak{m}_i \ \ \forall \ 1 \leq i \leq t &  ...(\text{note that } A/\mathfrak{m}_i \ \text{is perfect})
\end{align*}
Let $$ z= y+\delta^{\scr {{\rm char}( F)}}.$$
For later use, we note that 
$$ z\equiv 0 \ \syn{mod} \ \mathfrak{m}_i \ \ \ \forall \  1\leq i \leq t.$$
Using the fact that $z-y$ is ${\rm char}(F)$-th power of an element of $\mathfrak{m}$, it is straightforward to deduce the following from (1) and (2) of the above step. 
\begin{enumerate}
\item The map $\varphi: X\to \A^d_F$ defined by $(z_{\scr 1},...,z_{\scr d-1},z)$ is \'{e}tale at $p$. 
\item $\{z_{\scr 1},...,z_{\scr d-1},h(z)\}$ is an $F(p)$-basis of  $\mathfrak{m}/\mathfrak{m}^2$.
\item $z \ \syn{mod}\ \mathfrak{m}$ is the primitive element $c$ of $F(p)/F$.
\end{enumerate} 

We further claim that we have the following equality of  ideals of $\tilde{A}=A/(I(B))$ :
$$\sqrt{\left(\tilde{z}_{\scr 1},...,\tilde{z}_{\scr d-1},h(\tilde{z})\right)} = \tilde{\mathfrak{m}}.$$ To see the claim, we first observe  
\begin{align*}
h(z) & \in \tilde{\mathfrak{m}} \\
h(z) & \notin \tilde{\mathfrak{m}}_i \ \ \ \forall \  1\leq i \leq t.
\end{align*}
The first containment follows as $h$ is the irreducible polynomial of $z \ \syn{mod} \ \mathfrak{m}$. 
Moreover, since $h(0)\neq 0$, the second statement follows from the fact that $z\equiv 0 \ \syn{mod} \ \mathfrak{m}_i.$

As $\{\tilde{\mathfrak{m}},\tilde{\mathfrak{m}}_1,...,\tilde{\mathfrak{m}}_t\}$ are the only prime ideals of $\tilde{A}$ containing the ideal $(\tilde{z}_{\scr 1},...,\tilde{z}_{\scr d-1})$, and $h(z)\notin \mathfrak{m}_i \ \forall \ i$, we conclude that $\tilde{\mathfrak{m}}$ is the unique prime ideal of $\tilde{A}$ containing the ideal 
$\left(\tilde{z}_{\scr 1},...,\tilde{z}_{\scr d-1},h(\tilde{z})\right)$. Therefore 
$$\sqrt{\left(\tilde{z}_{\scr 1},...,\tilde{z}_{\scr d-1},h(\tilde{z})\right)} = \tilde{\mathfrak{m}}.$$

\noindent $\Step{4}$:  We claim that in fact  $$\big(\tilde{z}_{\scr 1},...,\tilde{z}_{\scr d-1},h(\tilde{z})\big) = \tilde{\mathfrak{m}}.$$
Note that both are $\tilde{\mathfrak{m}}$-primary ideals and hence it is enough to show the equality in the localization $\tilde{A}_{\tilde{\mathfrak{m}}}.$ But the equality holds in this local ring by Nakayama's Lemma since it holds modulo $\tilde{\mathfrak{m}}^2$  as $\{ z_{\scr 1},...,z_{\scr d-1},h(z)\}$ $\syn{mod} \ \mathfrak{m}^2$ gives a basis of $\mathfrak{m}/\mathfrak{m}^2$ (see condition (2) of the the above Step).  \\

\noindent $\Step{5}$: Recall that  $\varphi: X\to \A^d_F$ is the map defined by $(z_{\scr 1},...,z_{\scr d-1},z)$. We claim that $p$ is the unique point in $\varphi^{-1}\varphi(p)\intersection Z$. In fact we have that $p$ is the unique point of $\varphi^{-1}\varphi(p)\intersection B$. This  is a direct consequence of Step 3, since  the ideal defining $\varphi^{-1}\varphi(p)\intersection B$ in $B=\Spec(\tilde{A})$ is equal to $(\tilde{z}_{\scr 1},...,\tilde{z}_{\scr d-1},h(\tilde{z}))=\tilde{\mathfrak m}$.  Indeed, what we have observed is that the scheme $\varphi^{-1}\varphi(p)\intersection B$ is reduced and has $p$ as the only underlying point. Thus the same holds for $\varphi^{-1}\varphi(p)\intersection Z$. If $\mathfrak{n}$ denotes the maximal ideal in the coordinate ring of $\A^d_F$ of the point $\varphi(p)$, then $\mathfrak{n}\overline{A} = \overline{\mathfrak{m}}$.  Recall that $\overline{A}:=A/(f)$ and $Z=\Spec(\overline{A})$. \\

\noindent $\Step{6}$: In this step we prove the rest of the theorem using a trick used in the proof of \cite[3.5.1]{chk}.
In fact, the argument in this step has been directly taken from {\it loc. cit.}  The map $\varphi:Z\to \A^d_F$ is finite. Let  $\mathfrak{n}$ be the maximal ideal of $\varphi(p)$ in $F[X_1,...,X_d]$.  By Step 5, $\mathfrak{n}\overline{A}=\overline{\mathfrak{m}}$ and the map 
$$ \frac{F[X_1,...,X_d]}{\mathfrak{n}} \xrightarrow{\varphi^*} \frac{\overline{A}}{\mathfrak{n}{\overline{A}}} $$ is an isomorphism, in particular surjective. By Nakayama's lemma, there exists a $g\in F[X_1,...,X_d]\backslash \mathfrak{n}$ such that the map 
$$ F[X_1,...,X_d]_g \to \overline{A}_g $$ is surjective. In particular, if $V=\A^d_F\backslash \syn{Z}(g)$, then 
$$Z \intersection \varphi^{-1}(V)  \to V$$ is a closed immersion. Note that $V$ is an open neighbourhood of $\varphi(p)$. \\

\noindent Let $D\subset X$ be the maximal closed subset on which the map $\varphi$ is not \'{e}tale. Clearly $p\notin D$. Also, since $D$ is a subset of $B$ (see  Step 2) and $p$ is the only point in $\varphi^{-1}\varphi(p)\intersection B$, we must have $\varphi(p)\notin \varphi(D)$. However, the map $\varphi_{|B}$ is finite, we have that $\varphi(D)$ is a closed subset of  $\A^d_F$. Let 
$$ W:= \big( \A^d_F\backslash \varphi(D) \big) \intersection \big( \A^d_F\backslash \syn{Z}(g)\big). $$
Thus  $\varphi^{-1}(W)\to W$ is \'{e}tale. Moreover, $\varphi^{-1}(W)$ is an open neighbourhood of $p$. It is now clear that $\varphi$ and $W$ satisfy conditions (1) and (2) of the Lemma. Condition (3) is also immediate since the map $\varphi$ is defined by $(z_{\scr 1},...,z_{\scr d-1},z)$ and $z_{\scr i}$ vanish on $p$ for $1\leq i\leq d-1$.  
\end{proof}

\section{Reduction to open subsets of $\A^2_F$} \label{inductiond}

The previous section reduces the general case of Theorem \ref{gabberfinite} to the case where $X$ is an open subset of $\A^d_F$. The goal of this section is to further reduce to the case where $d=2$ (see Lemma \ref{dto2}). This reduction, which is an induction argument, is an important step in the proof of Theorem \ref{gabberfinite}. One of the ingredients required for this induction argument to work is the following  variation of the standard Noether normalization trick (see \eqref{nntrick1}).

\begin{lemma}\label{nntrick} Let $n\geq 2$ be any integer, $k$ be any field and 
 $Z/k$ be an affine variety of dimension $n-1$.  Let $$Z\xrightarrow{(\phi_1,...,\phi_n)}\A^{n}_k$$ be a finite map. Let  $Q \in k[\phi_{n}]$ be a non constant monic polynomial. Then for an integer $\ell>>0$, the map 
$$ Z\xrightarrow{(\phi_{\scr 1}-Q^\ell_{\scr 1},\ldots,\phi_{\scr n-1}-{Q^\ell_{\scr n-1}})} \A^{n-1}_k$$ is finite, where $Q_i$'s are inductively defined by 
$$ Q_{n-1}:= Q.$$
$$ Q_{i}:= \phi_{i+1}-Q_{i+1}^\ell \ \ \    \forall \ 1\leq i\leq n-2.$$
\end{lemma}
\begin{proof}The proof is  similar to that of \ref{nntrick1} (see \cite[page 2]{mum}) and hence we only give a sketch. Since $\dim(Z)= n-1$, $\phi_1, \phi_2,\ldots,\phi_n$ cannot be algebraically independent. Thus there exists a non-zero polynomial $f \in k[Y_1,...,Y_n]$  such that $f(\phi_1, \phi_2,\ldots,\phi_n)=0$.
Let $\ell$ be any integer greater than $n\dg(f)$ where $\dg(f)$ is the total degree of $f$. Let $\tilde{Q}\in k[Y_n]$ be a polynomial such that $Q=\tilde{Q}(\phi_n)$. Inductively define $\tilde{Q}_i$ for $1\leq i \leq n-1$ as follows:
\begin{align*}
 & \tilde{Q}_{n-1} : = \tilde{Q}  \\
 & \tilde{Q}_i  := Y_{i+1} - \tilde{Q}_{i+1}^\ell & \hspace{-10cm} \forall \ 1 \leq i \ \leq n-2. 
 \end{align*}
Notice that the polynomials $\tilde{Q}_i$'s are defined such that 
$$ \tilde{Q}_i(\phi_1,...,\phi_n) = Q_i.$$
Moreover, we note that if $d=Y_n$-degree of $\tilde{Q}$ then each $\tilde{Q}_i$ is monic in $Y_n$ of degree ${\ell}^{n-i-1}d$.
Consider the elements $Z_1,...,Z_{n-1}\in k[Y_1,...,Y_n]$ defined  as follows:
$$ Z_i := Y_i - \tilde{Q}_i^\ell \ \ \ \forall 1 \leq i \leq n-1.$$

We leave it to the reader to check that 
$$ k[Z_1,...,Z_{n-1},Y_n] = k[Y_1,...,Y_n].$$ For future reference, we note that the map $$\eta:\A^n_F\xrightarrow {(Y_1-\tilde{Q}_1^\ell,\ldots,Y_{n-1}-\tilde{Q}_{n-1}^\ell,Y_n)}\A^n_F $$ is an automorphism.
It is enough to show that the polynomial $f$, expressed in the variables $Z_1,...,Z_{n-1},Y_n$ is monic in $Y_n$. Let us write $f$ as  $$f=\Sigma_{I=(i_1,\ldots i_n)} \alpha_I \cdot m_{I}$$ where $m_I$'s are monomials in $Y_1,...,Y_n$ and $\alpha_I \in k$. 
  We leave it to the reader to verify that when expressed in new coordinates $Z_1,\ldots Z_{n-1},Y_n$, each monomial $m_I$ becomes a polynomial which is monic in $Y_n$ of $Y_n$-degree equal to $i_n+\Sigma_{k=1}^{n-1} i_k\cdot {\ell}^{n-k}\cdot d$.  Since $\ell > n \dg(f)$, one can show that these $Y_n$-degrees are distinct. Thus in the coordinates $Z_1,...,Z_{n-1},Y_n$, $f$ remains monic in $Y_n$. 
\end{proof}

\begin{notation} \label{not:openad}
Let $d\geq 2$ be an integer and $f,g\in F[X_1,...,X_d]$ be nonzero polynomials with no common irreducible factors (see Remark \ref{fgfactors}). Let $X:=\A^d_F\backslash \syn{Z}(g)$ and $Z:=\syn{Z}(f)\intersection X$. Let $p\in Z$ be a closed point (see Remark \ref{simplify}) whose first $d-1$ coordinates are $0$.  
\end{notation}

Recall that by \ref{redopen} it is enough to prove Theorem \ref{gabberfinite} for $(X,Z,p)$ as above. In order to prove this, we have to first come up with a map from $\Phi:X \to \A^d_F$. Indeed, we will look for maps $\Phi$ which are defined on the whole of $\A^d_F$. In other words, we will look for suitable polynomials $\{\phi_1,...,\phi_d\} \subset F[X_1,...,X_d]$. The goal of the following definition is to list necessary conditions on these polynomials which will ensure (see Lemma \ref{nakarg}) that the resulting map $\Phi$ is as desired in \eqref{gabberfinite}.

\begin{definition} \label{defpresents} Let $f,g, X, Z, p$ be as in Notation \ref{not:openad}. 
For $\{\phi_1,...,\phi_d\}\subset F[X_1,...,X_d]$, let
\begin{enumerate}
\item[(i)] $\Phi:\A^d_F\xrightarrow{(\phi_1,...,\phi_d)} \A^d_F$.
\item[(ii)] $\Psi:\A^d_F\xrightarrow{(\phi_1,...,\phi_{d-1})} \A^{d-1}_F$.
\end{enumerate}
We say that $(\phi_1,...,\phi_d)$ presents $(X,\syn{Z}(f),p)$ if 
\begin{enumerate}
\item $\Psi_{|\syn{Z}(f)}$ is finite and $\Psi(p)= (0,...,0)$. 
\item $\Psi^{-1}\Psi(p)\intersection \syn{Z}(f) \subset Z$ 
\item $\Phi$ is \'{e}tale at $S:=\Psi^{-1}\Psi(p)\intersection Z$. 
\item $\Phi$ is radicial at $S$. 
\end{enumerate}
\end{definition}

Recall that $\Phi$ is said to be {\em radicial} \cite[Tag 01S2]{stacks}  if $\Phi_{|S}$  is injective  and for all $x \in S$ the residue field extension $F(x)/F(\Phi( x))$ is trivial.\\

The following lemma shows that in order to prove Theorem \ref{gabberfinite} for $X,Z,p$ as in Notation \ref{not:openad}, it is enough to find $\phi_1,...,\phi_d$ which presents $(X,\syn{Z}(f),p)$. 
\begin{lemma}\label{nakarg} Let $X,Z,p$ be as above. Assume there exists $\{\phi_1,...,\phi_d\}$ which presents 
$(X,\syn{Z}(f),p)$ and $\Phi,\Psi$ be as in Definition \ref{defpresents}. Then there exist  open neighborhoods $V\subset \A^{d-1}_F$ of $\Psi(p)$ and $U\subset X$ of $p$, such that  $\Phi_{|U},\Psi_{|U},U, V$  satisfy conditions (1),(2),(3) of Theorem \ref{gabberfinite}.  Moreover, $\Psi^{-1}(V) \cap \syn{Z}(f)\subset U$.
\end{lemma}
\begin{proof} The argument here is directly taken from \cite[3.5.1]{chk}. We construct an open neighbourhood $V$ of  $\Psi(p)$ in  $\A^{d-1}_F$, such that 
if $Z_V:=\Psi^{-1}(V) \cap \syn{Z}(f)$ then 
\begin{enumerate}
 \item[(i)] $Z_V \subset Z$
 \item[(ii)] $\Phi$ is \'{e}tale at all points in $Z_V$ 
 \item[(iii)] $\Phi|_{Z_V}: Z_V \rightarrow \A^ 1_V $ is closed immersion
\end{enumerate}
Let $B$ be the smallest closed subset of $\syn{Z}(f)$ containing all points of $\syn{Z}(f)$ at which $\Phi$ is not \'{e}tale and also containing $\syn{Z}(f)\backslash Z$. Since $\Psi_{|\syn{Z}(f)}$ is a finite map, $\Psi(B)$ is closed in $\A^{d-1}_F$.  Moreover, because of conditions $(2)$ and $(3)$ of  Definition \ref{defpresents}, we have $\Psi(p) \notin \Psi(B)$. Thus, we can choose affine open subset $W \subset \A^{d-1}_F$ such that $\Psi(p) \in W \subset \A^{d-1}_F \backslash \Psi(B)$. Let $Z_W= Z\cap \Psi^{-1}(W)$. We have following commutative diagram of affine schemes and consequently their coordinate rings.

\begin{center}
\begin{minipage}{.4\textwidth}
\begin{tikzpicture}
 \node (E) at (0,0){$Z_W$};
    \node at (2,1) (F) {$\A^1_W$};
    \node at (2,-1) (A) {$W$};
    \draw[->] (E)--(F) node [midway,above] {$\Phi$};
    \draw[->] (F)--(A) node [midway,right] {$\pi$} ;
    \draw[->] (E)--(A) node [midway, left,below] {$\Psi$} ;
\end{tikzpicture} 
\end{minipage}
\begin{minipage}{.2\textwidth}
\begin{tikzpicture}
 \node (E) at (0,0){$ F[Z_W]$};
    \node at (2,1) (F) {$F[\A^1_W]$};
    \node at (2,-1) (A) {$F[W]$};
    \draw[->] (F)--(E) node [midway,above] {$\Phi^{*}$};
    \draw[->] (A)--(F) node [midway,right] {} ;
    \draw[->] (A)--(E) node [midway, left,below] {$\Psi^{*}$} ;
\end{tikzpicture}
\end{minipage}

\end{center}

Let $\Psi(p)=q$  and ${\mathfrak m}_q$ be the maximal ideal in $F[W]$ corresponding to $q$. Thus the ideal corresponding to $S=\Psi^{-1}(q)\intersection Z$ in $F[Z_W]$ is ${\mathfrak m}_q\cdot F[Z_W]$. Since $\Phi$ is radicial as well as \'{e}tale at $S$, $$\Phi_{|S}:S\inj \A^1_W$$ is a closed immersion. Thus the map on the coordinate rings 
$$F[\A^1_W] \twoheadrightarrow \frac{F[Z_W]}{{\mathfrak m}_q F[Z_W]}$$
is surjective. The surjectivity of the above map is equivalent to 
 $$ C\tensor_{F[W]} \frac{F[W]}{{\mathfrak m}_q}=0 $$
where $$C:={\rm Coker}\big(F[\A^1_W] \rightarrow F[Z_W]\big).$$ 
But $C$ is a finite $F[W]$ module. Hence by Nakayama's lemma $C_{{\mathfrak m}_q}=0$. Thus there exists $h \in F[W]\backslash {\mathfrak m}_q$ such that $C_h=0$ or equivalently 
$$ {F[\A^1_W]}_h \twoheadrightarrow {F[Z_W]}_h $$
is surjective.
 Let $V := W\backslash \syn{Z}(h)$. The coordinate ring of $Z_V:=\Psi^{-1}(V)\intersection \syn{Z}(f)$ is $F[Z_W]_h$ and that of $\pi^{-1}V$ is $F[\A^1_W]_h$. Thus the surjectivity of the above map implies that 
 $$ Z_V \inj \A^1_V$$ is a closed immersion as required. \\
 
Let $U\subset X$ be the maximal open subset containing points at which $\Phi$ is \'{e}tale.  To finish the proof, we need to show that $U,V, \Phi_{|U}, \Psi_{|U}$ satisfy conditions $(1),(2),(3)$ of Theorem \ref{gabberfinite}. (1) is clearly satisfied by the definition of $U$. To see (2), note that $\Psi_{|Z(f)}$ is finite, and hence, as $Z_V= \Psi^{-1}(V) \intersection Z(f)$,  $\Psi_{|Z_V}:Z_V  \to V$ is finite. (3) is precisely the condition (iii) mentioned at the beginning of the proof. By the construction of $W$, subsequently $V$, it follows that $\Psi^{-1}(V) \cap \syn{Z}(f) \subset U$.
\end{proof}

\begin{remark}\label{fgfactors} Since our main goal is to prove Theorem \ref{gabberfinite} for $(X,Z,p)$, we may change $\syn{Z}(f)$ as long as it does not change $Z$.  If $f$ and $g$ have common irreducible factors, dividing $f$ by the g.c.d. of $f$ and $g$ does not change $\syn{Z}(f) \backslash \syn{Z}(g)$. This justifies our assumption in Definition \ref{not:openad} that $f,g$ have no common irreducible factors. 
\end{remark}

The following lemma is proved using a simple coordinate change argument. It will be used in the proof of \eqref{dto2}, which is the main result of this section. 
\begin{lemma}\label{normalizev}
Let $(\phi_1,...,\phi_d)$ present $(X,\syn{Z}(f),p)$.  as in Lemma \ref{nakarg}. Then there exist $(\tilde{\phi}_1, ... ,\tilde{\phi}_d)$ which present $(X,\syn{Z}(f),p)$ such that there exists an open subset $V \subset \A^{d-1}_F$ satisfying the conclusion of Lemma \ref{nakarg} for $(\tilde{\phi}_1, ... ,\tilde{\phi}_d)$ and which satisfies the following additional condition: 
$$ \dim\bigg( \syn{Z}(\tilde{\phi}_1, ... ,\tilde{\phi}_{d-2}) \intersection \Psi^{-1}_{|Z(f)}\big( \A^{d-1}_F \backslash V \big)\bigg)=0.$$
\end{lemma}
\begin{proof} We note that if $d=2$, by convention, 
$$\syn{Z}(\tilde{\phi}_1, ... ,\tilde{\phi}_{d-2}) \intersection {\Psi^{-1}_{|Z(f)}}(\A^{1}_F \backslash V) = {\Psi^{-1}_{|Z(f)}}(\A^1_F\backslash V) $$
which is of zero dimension since $V$ is non-empty. Thus we may assume $d\geq 3$. 
For an integer $\ell$, consider the automorphism  $\rho: \A^{d-1}_F\to \A^{d-1}_F$ induced by 
$$ (X_1,...,X_{d-1}) \mapsto (X_1-X_{d-1}^{\ell^{(d-1)-1}}, X_2-X_{d-1}^{\ell^{(d-1)-2}}, ..., X_{d-2}-X_{d-1}^{\ell^{1}}, X_{d-1}). $$ 
We choose $\ell >>0$, such that by \eqref{nntrick1}, $(X_1,...,X_{d-2})_{| \rho(\A^{d-1}\backslash V)}$ is a finite map. 
Let
\begin{align*}
 \tilde{\phi}_i  & : = {\phi}_i - \tilde{\phi}_{d-1}^{\ell^{d-1-i}} \ \ \ \  \text{for } \  i\leq d-2 \\
 \tilde{\phi}_i & := \phi_i  \ \ \ \ \ \ \text{for }  i=d-1,d
\end{align*}
It is then straightforward to check that  $(\tilde{\phi}_1,...,\tilde{\phi}_d)$ presents $(X,\syn{Z}(f),p)$ (since it is obtained by a coordinate change from the original $\phi_i$'s) and moreover 
$$ \dim\bigg( \syn{Z}(\tilde{\phi}_1, ... ,\tilde{\phi}_{d-2}) \intersection {\Psi^{-1}_{|Z(f)}} (\A^{d-1}_F \backslash \rho(V)) \bigg)=0.$$
\end{proof}

\begin{lemma}\label{fgcommon} Let $d\geq 3$, and $f,g\in F[X_1,...,X_d]$ be two non-zero polynomials with no common factors. Let $p$ be a closed point of $\A^d_F$ such that $X_i(p)=0$ for all $i\leq d-1$. Then there exists a coordinate change of $F[X_1,...,X_d]$, i.e. elements $Y_i\in F[X_1,...,X_d]$ with $$ F[X_1,...,X_d] = F[Y_1,...,Y_d]$$ 
such that  $f(0,Y_2,...,Y_d)$ and $g(0,Y_2,...,Y_d)$ are nonzero polynomials with no common irreducible factors and $Y_i(p)=0$ for all $i\leq d-1$.  
\end{lemma}
\begin{proof}
The condition that $f(0,Y_2,...,Y_d)$ and $g(0,Y_2,...,Y_d)$ are nonzero polynomials with no common irreducible factors is equivalent to the condition that no irreducible component of $\syn{Z}(f)\intersection \syn{Z}(g)$ is contained in $\syn{Z}(Y_1)$. 

By Noether normalization trick \ref{nntrick1}, we may assume, by a suitable coordinate change, that the projection 
$$ (X_2,...,X_{d}) : \syn{Z}(f)\intersection \syn{Z}(g) \xrightarrow{\eta} \A^{d-1}_F$$ is finite. 
Note that since $d\geq 3$, the image of every irreducible component of $\syn{Z}(f)\intersection \syn{Z}(g)$ under $\eta$ is of dimension at least one. Thus we may choose closed points $z_1,...,z_\tau$, one in each irreducible component of $\syn{Z}(f)\intersection \syn{Z}(g)$ such that $\eta(z_i)$ are pairwise distinct and also different from $\eta(p)$. For every closed point $x$ of $\A^d_F$, either $X_1$ or $X_1+1$ is non-vanishing on $x$. Thus for each $z_i$, we choose $\epsilon_i=0$ or $1$, such that $X_1+\epsilon_i$ does not vanish on $z_i$. By Chinese remainder theorem, there exists a polynomial $\gamma\in F[X_2,...,X_{d}]$ such that 
$$ \gamma(\eta(z_i))= \epsilon_i \ \ \ \text{and} \ \ \ \gamma(p)=0.$$
It is now straightforward to check that $$Y_1:= X_1-\gamma \ \ \ \text{and} \ \ \ Y_i:=X_i \ \ \forall \ 2\leq i\leq d$$
satisfies our requirement. 
\end{proof}

\begin{lemma}\label{dto2} [Reduction to $d=2$] Assume that for $d=2$ and  every $f,g,X,Z,p$ as in Notation \ref{not:openad}, there exists $\phi_1,\phi_2\in F[X_1,X_2]$ which presents $(X,\syn{Z}(f),p)$. Then the same holds for every $d\geq 2$.  
\end{lemma}
\begin{proof}We prove this lemma by induction on $d$. Assume $d\geq 3$. \\

\noindent \underline{Step 0}: 
As before, we let $F[X_1,...,X_d]$ be the coordinate ring of $\A^d_F$. 
Let $\overline{f}(X_2,...,X_d):=f(0,X_2,...,X_d)$ and $\overline{g}(X_2,...,X_d):=g(0,X_2,...,X_d)$. By Lemma \ref{fgcommon}, we may assume that 
 then $\overline{f}$ and $\overline{g}$ are non-zero and have no common factors. 
We let 
\begin{enumerate}
\item[-] $\overline{X}:= X\intersection \syn{Z}(X_1)$.
\item[-] $\overline{Z}:=Z\intersection \overline{X}$.
\end{enumerate}
Note that $p\in \overline{Z}$ and $\overline{X}=\syn{Z}(X_1) \backslash \syn{Z}(g)$ where $\syn{Z}(X_1)\cong \A^{d-1}_F$ with coordinate ring $F[X_2,...,X_d]$. By induction, there exist $\{\overline{\phi}_2,...,\overline{\phi}_d\}\subset F[X_2,...,X_d]$ which presents $(\overline{X},\syn{Z}(\overline{f}),p)$. Let 
$$\overline{\Phi}:=(\overline{\phi}_2,...,\overline{\phi}_d) \ \ \text{and} \ \  \overline{\Psi}:=(\overline{\phi}_2,...,\overline{\phi}_{d-1}).$$ 
By Lemma \ref{nakarg}, there exist neighbourhoods $\overline{V}\subset \A^{d-1}_F$ and $\overline{U}\subset \overline{X}$ of $\overline{\Psi}(p)$ and $p$ respectively such that if 
$$ \overline{Z}_{\overline{V}}:= \overline{Z}\intersection \overline{\Psi}^{-1}(\overline{V})$$ 
then the following conditions of Theorem \ref{gabberfinite} 
\begin{enumerate}
\item $\overline{\Phi}_{|\overline{U}}$ is \'{e}tale 
\item $\overline{\Psi}_{|\overline{Z}_{\overline{V}}} : \overline{Z}_{\overline{V}} \to \overline{V}$ is finite 
\item $\overline{\Phi}_{|\overline{Z}_{\overline{V}} } : \overline{Z}_{\overline{V}} \to \A^1_{\overline{V}}$ is a closed immersion 
\end{enumerate}
{are satisfied.}\\
Further, by Lemma \ref{normalizev}, we also assume (without loss of generality) that if $E$ is the closed subset of ${Z(\bar{f})}$ defined by 
$$ {E:=Z(\bar{f})\backslash \overline{\Psi}^{-1}(\overline{V})} $$ 
then 
\begin{enumerate}
\item[(4)] $ \dim( E  \intersection \syn{Z}(\overline{\phi}_2,...,\overline{\phi}_{d-2}) ) = 0.$
\end{enumerate}
Note that (4) is vacuously satisfied unless $d\geq 4$. Indeed for $d = 3$,  $\A^{d-2}_{F}\setminus V$ is a finite set, and since $\Psi_{|Z(f)}:Z(\bar{f})\rightarrow \A^{d-2}_F$ is finite, $E$ is thus a
finite set.

\noindent $\Step{1}$: Since $\syn{Z}(\overline{f}) \xrightarrow{(\overline{\phi}_2,...,\overline{\phi}_{d-1})} \A^{d-2}_F$ is finite (see \ref{defpresents}(1)), for $2\leq i \leq d$, the image of $X_i$ in $F[X_2,...,X_d]/(\overline{f})$ satisfies a monic polynomial 
$$ P_i(T) := T^{m_i}+a_{m_i-1,i}T^{m_i-1} + \cdots + a_{0,i} $$
where each $a_{i,j} \in F[\overline{\phi}_2,...,\overline{\phi}_{d-1}]$. So $P_i(X_i)$ is zero in $F[X_2,...,X_d]/(\bar{f})$. Note that each $\overline{\phi}_i$ is an element of $F[X_2,...,X_d]$. Thus we have a map of algebras
$$F[\overline{\phi}_2,...,\overline{\phi}_{d-1}][T] \rightarrow  F[X_1,...,X_d][T]/(f).$$
We let $\tilde{P}_i(T)$ be the image of the  polynomial $P_i(T)$ under this map. 
Since $P_i(X_i)$ is zero in $F[X_2,...,X_d]/(\bar{f})$, $\tilde{P}_i(X_i)$ maps to zero via the map 
$$ F[X_1,...,X_d]/(f) \xrightarrow{X_1\mapsto 0} F[X_2,...,X_d]/(\overline{f}).$$
Therefore 
$$ \tilde{P}_i(X_i) = X_1g_i$$ for some $g_i\in F[X_1,...,X_d]/(f)$. We claim that the map 
$$ \syn{Z}(f) \xrightarrow{(\overline{\phi}_2,...,\overline{\phi}_{d},X_1,X_1g_2,...,X_1g_d)} \A^{2d-1}_F $$
is finite. This is clear because for $i\geq 2$, each $X_i$ satisfies the monic polynomial $\tilde{P}_i(T)-X_1g_i$  with coefficients which are polynomial expressions in the functions defining the above map. 
Applying \ref{nntrick1} repeatedly to this map (see Remark \ref{q=phin}), we get $\phi_2,...,\phi_d\in F[X_1,...,X_d]$ such that 
$$ \phi_i \equiv \overline{\phi}_i \ \ \syn{mod} \ X_1$$ 
and the map $(\phi_2,...,\phi_{d})_{|\syn{Z}({f})}$ is finite.

\noindent $\Step{2}$: Consider the maps 
\begin{enumerate}
\item[] $\widetilde{\Phi} :\A^d_F \xrightarrow{(X_1,\phi_2,...,\phi_d)} \A^{d}_F$
\item[] $\widetilde{\Psi}:\A^d_F \xrightarrow{(X_1,\phi_2,...,\phi_{d-1})} \A^{d-1}_F$. 
\end{enumerate}
Note that for all points $x\in \syn{Z}(X_1)$, $\widetilde{\Phi}$ is \'{e}tale at $x$ iff $\syn{Z}(X_1)\xrightarrow{\overline{\phi}_2,...,\overline{\phi}_{d}} \A^{d-1}_F$ is \'{e}tale at $x$. 
 Let $E$ be the  closed subset of $\syn{Z}(\overline{f})\subset \syn{Z}(f)$ defined in Step 0. We have the following: 
 \begin{enumerate}
\item  $\widetilde{\Phi}_{|\syn{Z}(f)}$ is finite. In fact, the map $(\phi_2,...,\phi_d)_{|\syn{Z}(f)}$ is finite. 
\item $\widetilde{\Psi}(p) \notin \widetilde{\Psi}(E)$ (this follows from the definition of $E$)
\item $\widetilde{\Phi}$ restricted to $\syn{Z}(\overline{f})\backslash E$ is a locally closed immersion. 
\item $\widetilde{\Phi}$ is \'{e}tale at all points in $\syn{Z}(\overline{f})\backslash E$.  
\end{enumerate}
By condition (4) of Step 0, $$E\intersection  \syn{Z}(\overline{\phi}_2,...,\overline{\phi}_{d-2})= E\intersection  \syn{Z}(\phi_2,...,\phi_{d-2}) \footnote{where by convention $\syn{Z}(\phi_2,...,\phi_{d-3})$ is the whole of  $\A^d_F$ if $d\leq 3$.} $$ is finite. 
Let $Q$ be any non-constant polynomial expression in $\phi_d$ which vanishes on the finite set 
$$\bigg(E\intersection  \syn{Z}(\phi_2,...,\phi_{d-2})\bigg)  \union \big\{p\big\}.$$ 
Let $\ell$ be a large enough integer which is divisible by $\Char(F)$. Let $\phi_1=X_1$ and 
as in Lemma \ref{nntrick}, let  $ Q_{d-1}:= Q$ and 
$$ Q_i : = \phi_{i+1}-Q_{i+1}^\ell \ \ \forall \ i\leq d-2 . $$
Let 
 $$ \Phi:=(\phi_1-Q_1^\ell,\ldots, \phi_{d-1}-Q_{d-1}^\ell,\phi_d): \A^d_F \longrightarrow \A^d_F$$ 
$$ \Psi:=(\phi_1-Q_1^\ell,\ldots, \phi_{d-1}-Q_{d-1}^\ell): \A^d_F \to \A^{d-1}_F.$$

By Lemma \ref{nntrick} $\Psi_{|\syn{Z}(f)}$ is finite. We let  
$S$ be the finite set of points in  $\Psi^{-1}\Psi(p)\intersection \syn{Z}(f)$.
To finish the proof,  it suffices to verify the conditions (2)-(4) of Definition \eqref{defpresents}. 
We first note that $S\subset \syn{Z}(\phi_1,...,\phi_{d-2})$.   This is because if $x\in S$, then by definition of $S$, 
$$ \phi_{i+1}-Q_{i+1}^\ell (x) = Q_{i}(x) = 0 \ \ \ \forall \ \ i\leq d-2.$$ 
And thus 
$$ \phi_i-Q_i^{\ell}(x) = \phi_i(x) = 0 \ \ \ \forall \ \ i\leq d-2.$$

We now show that $S$ is disjoint from $E$. First note that 
 $S \subset \syn{Z}(\phi_1)=\syn{Z}(X_1)$. Also {$\Psi(p)=0$ since $Q(p)=0$ and $\phi_i(p)=0$ for $1\leq i\leq d-1$}. Let $x\in S\intersection E$ if possible. Hence $x$ is necessarily in $E\intersection \syn{Z}(\phi_2,...,\phi_{d-2})$ by the above argument. In particular we note that $\phi_{d-2}(x)=0$.  Now we claim that 
 $$ \phi_{d-1}(x)=0.$$ 
 Since $\Psi(x)=0$ we have $(\phi_{d-2}-Q_{d-2}^\ell)(x)=0$. But as $\phi_{d-2}(x)=0$, we conclude that $$ Q_{d-2}(x)=0.$$
 Thus $$\phi_{d-1}(x)=(Q_{d-2}-Q^{\ell})(x)=0.$$
 This proves the claim. 
 Consequently, $x\in Z(\phi_2,\ldots\phi_{d-1})$.
By definition of $E$, $x \in E$ implies  $\bar{\Psi}(x)\notin \bar{V}$ where $\bar{V}$ is as defined in Step 0.   As $\bar{V}$ is a neighborhood of $0=\Psi(p)$, we have $\bar{\Psi}(x) \neq0$. But as $x\in E \subset Z(X_1)$, we have $$  \bar{\Psi}(x)= (\phi_2,...,\phi_{d-1})(x)=(\bar{\phi}_2,...,\bar{\phi}_{d-1})(x).$$   Hence $\phi_i(x)\neq 0 $ for some $ i$ with $2\leq i \leq d-1.$ This is a contradiction to the fact that $x\in Z(\phi_2,\ldots\phi_{d-1})$. 
Hence $S$ must be disjoint from $E$. Hence $ \tilde{\Phi}$ is a locally closed immersion on $S$ by property (3) of Step 2.

 As in the proof of Lemma \ref{nntrick}, we let  $$\A^{d-2}_F\xrightarrow{\ \ \eta\ \  } \A^{d-2}_F$$ be the automorphism defined by 
  $$\eta={(Y_1-\tilde{Q}_1^\ell,\ldots,Y_{d-1}-\tilde{Q}_{d-1}^\ell,Y_d)}$$ where  $\tilde{Q}_i\in F[Y_1,...,Y_d]$ are polynomials satisfying $Q_i= \tilde{Q}_i(\phi_1,...,\phi_{d})$. 
 It is straightforward to check that 
 $$ \Phi = \eta \circ \tilde{\Phi}. $$
 Hence $\Phi$ is a locally closed immersion on $S$, this proves condition (4) of Definition \ref{defpresents}.

 From Lemma \ref{nakarg} we have $Z(f) \cap \Psi^{-1}(V) \subset X$. This with the fact that $Z=Z(f)\cap X$ implies conditions (2) of Definition \ref{defpresents}. For checking condition (3), i.e. to check $\Phi$ is \'{e}tale at all points in $S$, we note that 
  since $\ell$ is divisible by $\Char(F)$, $\Phi$ is \'{e}tale precisely at those points where $\widetilde{\Phi}$ is \'{e}tale. In particular $\Phi$ is \'{e}tale at all points of $\syn{Z}(\overline{f})\backslash E$.  
\end{proof}


\section{Open subsets of $\A^2_F$}
In this section, we finish the proof of Theorem \ref{gabberfinite}. By Lemmas \ref{redopen}, \ref{dto2} we only have to deal with the case of open subsets of $\A^2_F$. While the handling of low degree points is similar, in spirit, to that of \cite{poonen}, for high degree points we proceed differently (see Lemma \ref{hexists}). 

\begin{lemma}\label{prime}
Let $F$ be a finite field as before, and  $C \subset \A^2_F$ be a closed  curve such that the projection 
onto the $Y$-coordinate $Y_{|C} : C\to \A^1_F$ is finite. Let $C^{(1)}$ denote the set of closed points of $C$. Then the following set of points is dense in $C$
	$$ \big\{ x \in C^{(1)} \ |  \ \dg_{{F}}(Y(x)) = \dg_{{F}}(x) \big\}.$$ 
\end{lemma}
\begin{proof} Without loss of generality, we may assume $C$ is irreducible and hence we simply have to show that the set $$ \big\{ x \in C^{(1)} \ | \ \dg_{{F}}(Y(x)) = \dg_{{F}}(x) \big\}$$ is infinite. Let $x_1,...,x_q$ be the $F$-rational points of $\A^1_F$. Let $C' : = C\backslash Y^{-1}(\{x_1,...,x_q\})$. $C'$ is a dense open subset of $C$ as $Y_{|C}$ is finite. Now, any point $x \in C'$ of prime degree satisfies $\dg_{{F}}(Y(x))=\dg_{{F}}(x)$.  By Lang-Weil estimates \cite{lw}, for all large enough prime number $\ell$, there is a point $x\in C'^{(1)}$ of degree $\ell$.  Hence, since $\ell$ is a prime, we must have $\dg_{{F}}(Y(x))=\dg_{{F}}(x)$. This proves the lemma. 
\end{proof}

\begin{notation}\label{notationa2} Let 
\begin{enumerate}
\item $A =F[X,Y]$ and for $d\geq 0$ let $A_{\scr \leq d} = \{h\in A\ |\ \dg(h)\leq d \}$. Here $\dg(h)$ denotes the total degree. 
\item $f,g\in A$ be two non-constant polynomials, with no common irreducible factors. By performing a change of coordinates if necessary, we will assume that  $f$ is monic in $X$ of degree $m$. 
\item $W:=  \A^2_F \backslash \syn{Z}(g)$. In this section, we call our variety $W$ instead of $X$, since the later will denote a coordinate function on $\A^2_F$. 
\item $Z:= \syn{Z}(f)\intersection W$. Note that $\syn{Z}(f)\backslash Z$ is finite as $f,g$ have no common irreducible components. 
\item $p\in Z$ be a closed point such that its $X$-coordinate is $0$. We also choose a set of closed points $\{p_1,...,p_t\}$ in $Z$ such that the set $T:= \{p,p_1,...,p_t\}$ satisfies 
	\begin{enumerate}
		\item $T$ contains at least one point from each irreducible component of $Z$.
		\item  No two points in $T$ have same degrees and for all $p_i\in T$,  $\dg(Y(p_i))= \dg(p_i)$. This can be ensured by Lemma \ref{prime}. Note that since $X$-coordinate of $p$ is $0$, we also have $\dg(Y(p))=\dg(p)$. 
	\end{enumerate}
\item Let $D=   \{q_1,\ldots, q_s\}$ be a finite set of closed points in $\syn{Z}(f)$ satisfying:
	\begin{enumerate} 
		\item  $D$ contains all points in $\syn{Z}(f)\backslash Z$. 
		\item $D$ contains at least one point from each irreducible component of $\syn{Z}(f)$. 
		\item $D$ does not contain any point of $\{p,p_1,...,p_t\}$.
	\end{enumerate}
\end{enumerate}
Moreover, for a point $x$ in $\syn{Z}(f)$, the notation $\sO_x$ (resp. ${\mathfrak m}_x$) will denote $\sO_{\A^2_F,x}$ (resp. ${\mathfrak m}_{\A^2_F,x}$) i.e.  the local ring (resp. maximal ideal) of $x$ as a point of $\A^2_F$. 
\end{notation}

The main result of this section is the following. 
\begin{theorem}\label{dtwo}
There exists $(\phi_1,\phi_2)\in F[X,Y]$ which presents $(W,\syn{Z}(f),p).$ 
\end{theorem}

This is enough to prove Theorem \ref{gabberfinite}. 
\begin{proof}[Proof of \ref{gabberfinite}]
This follows from Lemmas \ref{redopen}, \ref{nakarg} and \ref{dto2} and Theorem \ref{dtwo}. 
\end{proof}

To prove \ref{dtwo}, we will find $\phi_1$ by using Lemma \ref{phiexists} and $\phi_2$ by Lemma \ref{hexists}. We heavily use the counting techniques by Poonen \cite{poonen} to prove these lemmas.

Recall from \ref{notation}, for $Y$ a subset of a scheme $X/F$,  $Y_{ \scr \leq r}:= \{x\in Y\ |\ \dg(x) \leq r\}$. 

\begin{lemma}\label{phiexists}
 Let the notation be as in \ref{notationa2}. There exists  $c\in \N$, such that for every $d>>0$, there exists a $\phi \in A_{\scr \leq d}$  satisfying 
 \begin{enumerate}
  \item $\phi(p)=\phi(p_i)=0$ for all $i=1,...,t$ and  $ \phi(q_i) \neq 0 $ for all $i=1,\ldots, s$.
  \item $ (\phi,Y)$ is \'{e}tale at all $x\in S$ where $S:=\syn{Z}(\phi) \cap Z$.
  \item The projection $Y:\A^2_F \to \A^1_F$ is radicial at $S_{\scr \leq {(d-c)}/{3}}$.  
  \end{enumerate}
 \end{lemma}

\begin{remark}\label{strategy} The above lemma is motivated by writing down conditions for $\phi$ such that $(\phi,Y)$ presents $(W,\syn{Z}(f),p)$, and then keeping only those which we can prove. Indeed, if  $\phi_{\scr | Z(f)}$ is a finite map and $Y$ is radicial at whole of $S$ (as opposed to $S_{\scr (d-c)/{3}}$ above), then $(\phi,Y)$ would present $(W,Z,p)$ thereby proving \eqref{dtwo}. 
\end{remark}

\begin{remark}\label{sfinite}
The set $S=\syn{Z}(\phi) \cap Z$ appearing in the statement of the above Lemma is necessarily finite. This is because, in each irreducible component of $Z$, there is at least one $q_i$ (see \eqref{notationa2}(6)(b)) on which $\phi$ does not vanish. Since $T$ intersects each irreducible component of $\syn{Z}(f)$ (see \eqref{notationa2}(5)(a)), we know that any open neighbourhood of $S$ is dense in $\syn{Z}(f)$. 
\end{remark}

Following \cite{poonen}  define the density of a subset $\sC \subset A$ by $$\mu(\sC) := \displaystyle	\lim_{d \rightarrow \infty} \frac{\hash (\sC\cap A_{\scr \leq d})}{\hash  A_{\scr \leq d}}$$ provided the limit exists. Similarly, the upper and lower densities of $\sC$, denoted by $\overline{\mu}(\sC)$ and $\underline{\mu}(\sC)$, are defined by replacing limit in the above expression by  lim sup and lim inf respectively. \\

To prove the existence of $\phi$ in \ref{phiexists}, we will show that the density of such $\phi$ is positive. 
  We prove  Lemma \ref{phiexists} in two steps. First, we show (Lemma \ref{findg}) that $\phi$ satisfying conditions $(1),(3)$ and  condition $(2)$ for points upto certain degree, exists. Next, we show (Lemma \ref{findh})  that the set of $\phi$ which does not  satisfy condition (2) for points of higher degrees has  zero density. \\

Let $\phi\in A$ and $r\geq 1$ be an integer. Consider the following conditions on $\phi$, which are closely related to the conditions (1), (2) and (3) of Lemma \ref{phiexists}.
\begin{enumerate}
\item[(a)] $\phi(p)=\phi(p_i)=0$ for all $1 \leq i\leq t$ and $\phi(q_i)=1$ for all $1\leq i\leq s$.
\item[($b_r$)] For all $x\in \syn{Z}(f)_{\scr \leq r}$ such that $\phi(x)=0$, $\frac{\partial \phi}{\partial X}(x)\neq 0$. 
\item[($c_r$)] For all points $x_1,x_2 \in \syn{Z}(f)_{\scr \leq r}$, such that  
$ \dg(x_1)=\dg(x_2)=\dg(Y(x_1))=\dg(Y(x_2))$ and $\phi(x_1)=\phi(x_2)=0$, we have $Y(x_1)\neq Y(x_2)$. 
\item[($d_r$)] For all $x \in \syn{Z}(f)_{\scr \leq r}$ such that $\phi(x)=0$, $\dg(Y(x))=\dg(x)$. 
\end{enumerate}

It is easily seen that 
\begin{remark} \label{conditions} The main motivation for introducing the above conditions, are the following straightforward implications between them and the conditions of \ref{phiexists}
\begin{enumerate}
 \item[-] $\phi$ satisfies \eqref{phiexists}(1) if $\phi$ satisfies (a).
\item[-] $\phi$ satisfies \eqref{phiexists}(2) iff $\phi$ satisfies ($b_r$) for all $r\geq 1$.
\item[-] $\phi$ satisfies \eqref{phiexists}(3) iff $\phi$ satisfies ($c_r$) and ($d_r$) for all $r\leq (d-c)/{3}$. 
\end{enumerate}
\end{remark}

\begin{lemma}\label{findg} 
There exists integers $r_0, c\in \N$, with $$r_0> {\rm max} \big\{ \dg(p), \dg(p_1),...,\dg(p_t), \dg(q_1),..., \dg(q_s)\big\} $$ such that the lower density of the set 
$$
 \sP:=\Union_{d>c+2r_0}  \Big\{\phi \in A_{\scr \leq d}\ |\  \phi \ {\rm satisfies}\  (a), (b_{\scr (d-c)/{3}}), (c_{\scr (d-c)/{3}}), (d_{\scr (d-c)/{3}}) \ {\rm and} \ \phi(x)=1 \ \forall \   x\in \syn{Z}(f)_{\scr \leq r_0} \backslash T  \Big\} \\
   $$   
 is positive.  
 \end{lemma}
 \begin{proof}
By Lang-Weil estimates \cite{lw}  there exists $c'\in \N$  such that for all $n\geq 1$, 
$$ \hash  \big( \syn{Z}(f)_{\scr = n} \big) \leq c'\cdot q^{n}.$$

 For reasons which be clear during the course of the calculations below, we choose $r_0$ and $c$ as follows. {Recall that $m$ is the $X$-degree of $f$.} Let $r_0$ be any integer satisfying 
\begin{enumerate}
\item[(i)] $r_0  > {\rm max} \big\{ \dg(p), \dg(p_1),...,\dg(p_t), \dg(q_1),..., \dg(q_s)\big\}$.
\item[(ii)] $\displaystyle{\Bigg( \sum _{i>r_0/m} \frac{1}{q^{i}} \Bigg)\cdot \Bigg( c'+ \binom{m}{2} + \frac{m}{2} \Bigg)  <1-\sum_{x\in T}q^{- \dg(x)}.}$ 
\end{enumerate}
Note that it is always possible to ensure (ii) as $$ \Bigg(\sum _{i>r_0/m} \frac{1}{q^{i}}\Bigg) \to 0 \ \ \ \text{as} \ \ \ r_0\to \infty $$
and as degrees of points in $T$ are distinct we have 
$$ \sum_{x\in T}q^{-\dg(x)} < \sum_{i=1}^\infty q^{-i} \leq 1.$$ 
Let 
$$ c = \sum_{x\in \syn{Z}(f)_{\scr \leq r_0}} \dg(x). $$
     
Let $d\geq c+2r_0$ be any integer and $r:=(d-c)/{3}$. Let 
\begin{align*}
\sT & := \big\{ \phi \in A_{\scr \leq d} \ | \ \phi \mbox{ satisfies (a) } \ \text{and} \ \phi(x)=1 \ \forall \   x\in \syn{Z}(f)_{\scr \leq r_0} \backslash T \big\}. \\
\sT_b & := \big\{\phi \in \sT \ | \ \phi \ \text{does not satisfy} (b_{r}) \big\} \\
\sT_c & := \big\{\phi \in \sT \ | \ \phi \ \text{does not satisfy} (c_{r}) \big\} \\
\sT_d & := \big\{\phi \in \sT \ | \ \phi \ \text{does not satisfy} (d_{r}) \big\} 
\end{align*}
Let 
$$\delta  := \frac{\hash   \sT }{\hash  A_{\scr \leq d}}, \ \ 
\delta_b  := \frac{\hash \sT_b}{\hash  A_{\scr \leq d}}, \ \ 
\delta_c  := \frac{\hash \sT_c}{\hash  A_{\scr \leq d}}, \ \ 
\delta_d  := \frac{\hash \sT_d}{\hash  A_{\scr \leq d}}. 
$$ 
\\ \\
In the following steps we will estimate $\delta,\delta_b,\delta_c,\delta_d$. \\  \\
\noindent $\Step{1}:$ (Estimation for $\delta$) :
Note that the condition that $\phi$ belongs to $\sT$ depends solely on the image of $\phi$ in the zero dimensional ring 
$$ \prod_{x\in \syn{Z}(f)_{\leq{r_0}}} (\sO_x/m_x). $$ 
Since the {dimension over $F$} of the above ring is $c$ and since $d\geq c$, by \cite[Lemma 2.1]{poonen} the map 
$$ A_{\scr \leq d} \stackrel{\rho}{\longrightarrow} \prod_{x\in \syn{Z}(f)\leq{r_0}} (\sO_x/m_x)$$
is surjective. One can easily see that $\sT$ is a coset of $\Ker(\rho)$. Therefore 
$$ \delta = \prod_{x\in \syn{Z}(f)_{\scr \leq r_0}}q^{-\dg(x)}.$$   \\

\noindent $\Step{2}:$ (Estimation for $\delta_b$) :
Let $x \in \syn{Z}(f)_{\scr \scr \leq r}$ where recall that $r=(d-c)/{3}$.  The following are equivalent :
\begin{enumerate}
\item[(i)] $\phi\in \sT$ and $\phi(x)=0$ and $\frac{\partial \phi}{\partial X}(x)=0$.
\item[(ii)] $\phi  \in \sT$ and $\phi \ \syn{mod} \ {\mathfrak m}_x^2$ lies in the kernel of the linear map $\frac{\partial}{\partial X}: \frac{{\mathfrak m}_x}{{\mathfrak m}_x^2} \to F(x).$ 
\end{enumerate}

Let us first consider the case when $\dg(x)>r_0$. In this case, 
each of the above condition for $\phi$ depends only on its image in the zero dimensional ring 
$$ \Bigg(\prod_{\substack{ q\in \syn{Z}(f)_{\scr \leq{r_0}}} } (\sO_q/m_q) \Bigg)\times (\sO_x/m_{x}^2) .$$ 
The cardinality of the above ring is $$\bigg(\prod_{y\in \syn{Z}(f)_{\scr \leq r_0}} q^{\dg(y)}\bigg) \cdot q^{3 \dg (x)}.$$
Let us call an element $\xi$ in the above ring as a favorable value iff all $\phi$ mapping to $\xi$ satisfy the above conditions. It is an easy exercise to check that the set of all favorable values has cardinality $q^{\dg(x)}$. 
Thus the ratio of the number of favorable values to the cardinality of the ring is nothing but $\delta q^{-2 \dg(x)}$. 
 The {dimension over $F$} of this ring is $c+3 {\cdot \dg(x)}$. Since $d\geq c+3{\cdot\dg(x)}$, by \cite[Lemma 2.1]{poonen}, $A_{\scr \leq d}$ surjects onto this ring. Due to this, the ratio of $\phi \in A_{\scr \leq d}$ satisfying the above two conditions to the $\# A_{\scr \leq d}$ is nothing but $\delta q^{-2 \dg(x)}$. 
In other words, 
$$ \frac{ \hash  \big\{ \phi \in \sT \ | \ \phi(x)=0, \ \frac{\partial \phi}{\partial X}(x)=0 \big\}}{\hash  A_{\scr \leq d}} = \delta \cdot q^{-2 \dg(x)}.$$ 

Now let us consider the case where $\dg(x)\leq r_0$. We claim that unless $x\in T$, there is no $\phi \in \sT$ which vanishes on $x$. This follows from the definition of $\sT$. So let us assume $ x\in T$. In this case, the above two conditions for $\phi$ depend solely on the image of $\phi$ in the ring 
$$ \Bigg(\prod_{\substack{ q\in \syn{Z}(f)_{\scr \leq{r_0}}\\ q\neq x }  } (\sO_q/m_q) \Bigg)\times (\sO_x/m_{x}^2) .$$ 
Proceeding in a manner similar to the case where $\dg(x)>r_0$, we find that for $x\in T$, 
$$ \frac{ \hash  \big\{ \phi \in \sT \ | \ \phi(x)=0, \ \frac{\partial \phi}{\partial X}(x)=0 \big\}}{\hash  A_{\scr \leq d}} = \delta \cdot q^{- \dg(x)}.$$

Since 
$$ \sT_b = \Union_{\substack{x \in \syn{Z}(f)_{\scr \leq r}  \text{ such that} \\ \ x\in T \ {\text or} \ \dg(x)> r_0}} \hspace{-5mm} \big\{ \phi \in \sT \ | \ \phi(x)=0, \ \frac{\partial \phi}{\partial X}(x)=0 \big\}$$

we get an estimate 
\begin{align*} \delta_b & \leq      \sum_{x\in T} \delta q^{- \dg(x)}  + \sum_{\substack{x\in \syn{Z}(f)_{\scr \leq r} \text{ such that} \\\dg(x)>r_0}} \delta \cdot q^{-2 \dg(x)}  \\
  & \leq  \delta \bigg(  \sum_{x\in T}  q^{- \dg(x)}+ \sum_{r_0<i\leq r} c'q^{-i}\bigg) 
\end{align*}\\
where recall that $c'$ was the constant in Lang-Weil estimates such that $\# \syn{Z}(f)_{=n} \leq c' q^n$. \\

\noindent $\Step{3}:$ (Estimation for $\delta_c$): Let $y\in \A^1_F$ with $i:=\dg(y)\leq r$. Let
$$ \sT_c^y := \Big\{ \phi\in \sT \ | \ \exists \ \text{distinct} \  x_1, x_2 \in \syn{Z}(f)_{=i} \ \text{with}\  Y(x_1)=Y(x_2)=y \ \text{and}  \ \phi(x_1)=\phi(x_2)=0\Big\}.$$
First, note that $\sT_c^y$ is empty unless $i>r_0$. This is because the only points of degree $\leq r_0$ on which a $\phi \in \sT$ vanishes are the points in $T$. However, by choice, all points in $ x \in T$ have different degrees and satisfy $\deg(x)=\deg(Y(x))$. Thus, let us assume $i>r_0$. 
In this case, we claim that 
$$ \frac{ \hash  \sT_c^y } {\hash  A_{\scr \leq d}} \leq \delta \cdot {m\choose2}\cdot q^{-2i}.$$
For fixed $x_1,x_2$ with $Y(x_1)=Y(x_2)=y$, 
$$ \Big\{ \phi\in \sT \ | \ \phi(x_1)=\phi(x_2)=0  \Big\}$$
 is a coset of the kernel of the following  map 
$$ A_{\scr \leq d} \longrightarrow \bigg( \prod_{q\in \syn{Z}(f)_{\scr \leq{r_0}}} (\sO_q/m_q) \bigg) \times (\sO_{x_1}/m_{{x_1}}) \times (\sO_{x_2}/m_{{x_2}}) $$
which is surjective by \cite[2.1]{poonen}. Thus  
$$ \frac{ \hash  \Big\{ \phi\in \sT \ | \ \phi(x_1)=\phi(x_2)=0.\Big\}} {\hash  A_{\scr \leq d}} \leq \delta \cdot q^{-2i}.$$
To prove the claim we now simply observe that 
since $f$ is monic in $X$ of degree $m$  there are atmost $m\choose{2}$ possible choices for a pair $\{x_1,x_2\}$ as above. \\

\noindent  As discussed above, since $\sT_c^y$ is empty unless $i>r_0$, we have  
$$ \sT_c = \Union_{\substack{y\in \A^1_F \\ r_0< \dg(y)\leq r}} \sT_c^y.$$
For a fixed $i$, 
$$ \hash \big\{ y\in \A^1_F \ | \ \dg(y)=i\big\} \leq q^i.$$
 From this, it is elementary to deduce  
$$ \delta_c=\frac{\hash  \sT_c}{\hash A_{\scr \leq d}} \leq \delta \Big(  \sum_{r_0<i\leq r} {m \choose 2}  q^{-i}\Big).$$\\

\noindent $\Step{4}:$ (Estimation for $\delta_d$): As in the above step, let $y\in \A^1_F$ with $i:=\dg(y)\leq r$. Let
$$ \sT_d^y := \Big\{ \phi\in \sT \ | \ \exists \ x\in \syn{Z}(f)_{\scr \leq r} \ \text{with} \ \phi(x)=0 \ , \ Y(x)=y  \ \text{and}\ \dg(x)> i .\Big\}.$$

We first claim that $\sT_d^y$ is empty unless $\dg(y)>r_0/m$.  Otherwise, there would exist a $\phi\in \sT$ and an $x\in \syn{Z}(f)_{\scr \leq r}$ with $Y(x)=y$, $\phi(x)=0$ and $\dg(x)> \dg(y)$. But as $f$ is monic in $X$ of degree $m$, the maximum degree of a point $x$ lying over $y$ is $m \cdot \dg(y)\leq r_0$. Which means $x\in \syn{Z}(f)_{\scr \leq r_0}$. However as $\phi \in \sT$,  the only points in $\syn{Z}(f)_{\scr \leq r_0}$ on which  $\phi$ vanishes are those in $T$. Thus $x\in T$. But by \eqref{notationa2}(5)(c), for such $x$, $\dg(Y(x))=\dg(y)=\dg(x)$ which is a contradiction. \\

 We will now estimate 
$$ \frac{ \hash  \sT_d^y } {\hash  A_{\scr \leq d}}.$$

Fix a point  $x\in \syn{Z}(f)_{\leq r}$ with $\dg(x)>i$ and $Y(x)=y$. For this $x$, we first note that 
because of \eqref{notationa2}(5)(c), $x\notin T$. 

For $\dg(x)>r_0$ we note that  
$$ \frac{ \hash  \Big\{ \phi\in \sT \ | \ \phi(x)=0.\Big\}} {\hash  A_{\scr \leq d}} \leq \delta \cdot q^{-\dg(x)}\leq \delta \cdot q^{-2i}.$$
This is deduced, as before, from the surjectivity of  
$$ A_{\scr \leq d} \longrightarrow \bigg(\prod_{q\in \syn{Z}(f)_{\scr \leq{r_0}}} (\sO_q/m_q) \bigg)\times (\sO_{x}/m_{{x}})$$

If $\dg(x)\leq r_0$,  $\Big\{ \phi\in \sT \ | \ \phi(x)=0\Big\}$ is empty as there is no points in $\syn{Z}(f)\backslash T$ on which a $\phi \in \sT$ vanishes.  And hence, the above estimate trivially holds in this case also. 

As $f$ is monic in $X$ of degree $m$, and $\dg(x)\geq 2i$, there are at most $\frac{m}{2}$ possible choices for  $x\in \syn{Z}(f)$ such that $Y(x)=y$. This shows that 
$$ \frac{ \hash  \sT_d^y } {\hash  A_{\scr \leq d}} \leq \delta \cdot \frac{m}{2}\cdot q^{-2i}.$$

Since, as discussed above, $\sT_d^y$ is empty unless $\dg(y)>r_0/m$, we have
$$ \sT_d = \Union_{y\in (\A^1_F)_{\scr \geq r_0/m}} \sT_d^y.$$

For a fixed $i$, 
$$ \hash \big\{ y\in \A^1_F \ | \ \dg(y)=i\big\} \leq q^i.$$
 Thus 
$$ \delta_d=\frac{\hash  \sT_d}{\hash A_{\scr \leq d}} \leq \delta \Big(  \sum_{r_0/m<i\leq r} \frac{m}{2}\   q^{-i}\Big).$$
\\
\noindent $\Step{5}:$ (Estimation for $\sP$):
If we let 
$$ \sP_d := \Big\{ \phi \in A_{ \scr \leq d} \ | \  \phi \ {\rm satisfies}\  (a), (b_{\scriptscriptstyle (d-c)/{3}}), (c_{\scriptscriptstyle (d-c)/{3}}), (d_{\scriptscriptstyle (d-c)/{3}}) \ {\rm and} \ \phi(x)=1 \ \forall \   x\in \syn{Z}(f)_{\scr \leq r_0} \backslash T   \Big\},$$
then $$\sP_d = \sT \backslash(\sT_b\cup \sT_c \cup \sT_d). $$
Therefore 
\begin{align*}
\frac{\hash  \sP_d}{\hash  A_{\scr \leq d}} & \geq \delta - \delta_b- \delta_c -\delta_d \\
     & \geq \delta \Bigg[1 - \sum_{x\in T}q^{- \dg(x)} - \sum_{r_0<i\leq r} c'q^{-i}- \sum_{r_0<i\leq r} {m \choose 2}  q^{-i} - \sum_{r_0/m<i\leq r} \frac{m}{2}\   q^{-i} \Bigg] \\
     & \geq \delta \Bigg[ 1- \sum_{x\in T}q^{- \dg(x)} - \Bigg( \sum _{r_0/m<i \leq r} \frac{1}{q^{i}} \Bigg)\cdot \Bigg( c'+ \binom{m}{2} + \frac{m}{2} \Bigg) \Bigg]     
\end{align*}
Note that in the above expression $r=(d-c)/{3}$. As $d\to \infty$, so does $r$. Hence we observe that 
$$ \syn{inf} \Big( \frac{\hash  \sP_d}{\hash A_{\scr \leq d}} \Big) \geq \delta \Bigg[ 1- \sum_{x\in T}q^{- \dg(x)}- \Bigg( \sum _{i>r_0/m} \frac{1}{q^{i}} \Bigg)\cdot \Bigg( c'+ \binom{m}{2} + \frac{m}{2} \Bigg) \Bigg]  $$
which is positive, thanks to the definition of $r_0$. Thus the lower density of 
$$\sP= \Union_d \sP_d$$ is positive as required.       
  \end{proof}
 
\begin{lemma} \label{findh} Let $c$ be as in Lemma \ref{findg} and let $$\sQ:=\bigcup_{d\geq 0} \Big\{\phi\in A_{\scr \leq d} \ |\ \exists\  x\in \syn{Z}(f)_{\scr >  (d-c)/{3}} \ {\rm  such \ that } \ \phi(x)= \frac{\del \phi}{\del X}(x)= 0\Big\}.$$ Then ${\mu}(\sQ)=0.$
\end{lemma}
\begin{proof} The proof is identical to that of \cite[2.6]{poonen}. We reproduce the argument verbatim here for the convenience of the reader. We will bound the probability of  $\phi$ constructed as 
 $$\phi=\phi_0+g^p X+h^p$$ and for which there is a point $x\in \syn{Z}(f)_{\scr >  (d-c)/{3}}$ with $\phi(x)= \frac{\del \phi}{\del X}(x)= 0$.  Note that if $\phi$ is of the above form, then $$ \frac{\partial \phi}{\partial X}=\frac{\partial \phi_0}{\partial X}+g^p.$$ Further, if 
$\phi_0\in A_{\scr \leq d}$, $g \in A_{\scr \leq d-1/p}$ 
 and $h\in A_{\scr \leq d/p}$, then $\phi\in A_{\scr \leq d}$.
  Define $$W_0:=\syn{Z}(f) \ \ \ \text{and} \ \ \ W_1:=\syn{Z}\Big(f,\frac{\del \phi}{\del X}\Big).$$ 
Note that $\syn{dim}(W_0)=1$.   \\
Let  $$ \gamma:= \rdown{\frac{d-1}{p}} \ \ \ \text{and} \ \ \ \eta= \rdown{\frac{d}{p}}.$$

\noindent Claim 1:  The probability (as a function of $d$) of choosing $\phi_0 \in A_{\scr \leq d}$ and $g\in A_{\scr \leq (d-1)/p}$   such that $\syn{dim}(W_1)=0$ is $1-o(1)$ as $d\to \infty$.\\ \\
 Let $V_1,...,V_{\ell}$ be $F$ irreducible components of $W_0$. Clearly $\ell\leq \dg(f)$ (where $\dg(f)$ is the total degree). Since the projection onto the $Y$ coordinate is finite on $\syn{Z}(f)$ (by \eqref{notationa2}(2)), we know that $Y(V_k)$ is of dimension one for all $k$. 
We will now bound the set 
$$ G_k^{\rm bad} := \Big\{ g \in A_{\scr \leq \gamma} \ | \ \frac{\partial \phi}{\partial X}= \frac{\partial \phi_0}{\partial X}+g^p \ \text{vanishes identically on}\  V_k \Big\}.$$
If $g,g'\in G_k^{\rm bad}$, then $g-g'$ vanishes on $V_k$. Thus if $G_k^{\rm bad}$ is non-empty, it is a coset of the subspace of functions in $A_{\scr \gamma}$ which vanish identically on $V_k$. The codimension of that subspace, or equivalently the dimension of the image of $A_{\scr \gamma}$ in the regular functions on $V_k$ is at least $\gamma+1$, since no polynomial in $Y$ vanishes on $V_k$. Thus the probability that $\frac{\partial \phi}{\partial X}$ vanishes on $V_k$ is at most $q^{-\gamma-1}$. Thus, the probability that $\frac{\partial \phi}{\partial X}$ vanishes on some $V_k$ is at most $\ell q^{-\gamma-1}=o(1)$.  Since $\dim(W_1)=0$ iff $\frac{\partial \phi}{\partial X}$ does not identically vanish on any component $V_k$, the claim follows. \\

We will now estimate the probability of choosing $h$ such that there is no bad point in $Z(f)$, i.e. a point in $\syn{Z}(f)_{\scr > (d-c)/{3}}$ where both $\phi$ and $\frac{\partial \phi}{\partial X}$ vanish. Note that the set of such bad points is precisely 
$$ \syn{Z}(\phi)\intersection W_1 \intersection \syn{Z}(f)_{\scr > (d-c)/{3}}.$$
\noindent Claim 2: Conditioned on the choice of $\phi_0$ and $g$ such that $W_1$ is finite, the probability of choosing $h$ such that $$ \syn{Z}(\phi)\intersection W_1 \intersection \syn{Z}(f)_{\scr > (d-c)/{3}}=\emptyset$$ is 
$1-o(1)$ as $d\to \infty$. \\ \\ It is clear by the Bezout theorem that $\hash   W_1 = O(d)$. For a given $x\in W_1$, the set $$H^{\rm bad} = \big\{ h\in A_{\eta} \ | \  \phi=\phi_0+g^pX+h^p \text{ vanishes on } x \big\} $$
is either $\emptyset$ or a coset of $\Ker\big( A_{\scr \eta} \xrightarrow{ev_x} F(x)\big)$ where $F(x)$ is the residue field of $x$. For the purpose of this claim, we only need to consider $x$ such that  $\dg(x)>(d-c)/{3}$. In this case,  \cite[Lemma 2.5]{poonen} implies that 
$$ \frac{\hash H^{\rm bad}}{\hash A_{\scr \eta}} \leq q^{-\nu}  \ \ \ \text{where } \ \nu=\syn{min}(\eta+1,(d-c)/{3}).$$ Thus, the probability that both $\phi$ and $\frac{\partial \phi}{\partial X}$ vanish at such $x$ is at most $q^{-\nu}$. There are at most $\hash W_1$ many possibilities for $x$. Thus the probability that there exists a 'bad point', i.e. point in $x\in W_1$ with $\dg(x)> (d-c)/{3}$ such that both $\phi$ and $\frac{\partial \phi}{\partial X}$ vanish at such $x$ is at most $\big(\hash W_1\big)q^{-\nu} = O(dq^{-\nu})$. Since as $d\to \infty$, $\nu$ grows linearly in $d$, 
$O(dq^{-\nu})=o(1)$. In other words, the probability of choosing $h$ such that there is no bad point is $1-o(1)$. \\

Combining the above two claims, it follows that the probability of choosing $\phi=\phi_0+g^pX+h^p$ such that $$\syn{Z}(\phi)\intersection W_1 \intersection \syn{Z}(f)_{\scr > (d-c)/{3}}=\emptyset$$ is equal to 
$(1-o(1))(1-o(1))=1-o(1)$. This shows that $\mu(Q)=0$. 
 \end{proof}
  
 \begin{proof}[Proof of Lemma \ref{phiexists}]
 Let $\overline{Q}$ denote the complement of $Q$ in $A$. Let $\sP$ be as in Lemma \ref{findg}. To prove Lemma \ref{phiexists} we need to show that $\sP\intersection \overline{Q}$ is non-empty. However, combining the above two lemma's, we in fact get that 
 $\mu(\sP\intersection \overline{Q})>0$. This finishes the proof. 
 \end{proof}
 
  Condition (3) of Lemma \ref{phiexists} ensures that $Y:\A^2_F\to \A^1_F$ is radicial at $S_{\scr \leq  (d-c)/{3}}$.  We would have ideally liked to have $S$ instead of $S_{\scr \leq  (d-c)/{3}}$ here. If this was the case, and if $\phi_{ |Z(f)}$ was finite, as mentioned in Remark \ref{strategy}, we would be able to deduce that $(\phi,Y)$ presents $(W,\syn{Z}(f),p)$.  However we are unable to handle points in $S$ of degree greater than $ (d-c)/{3}$. In order to rectify that, we replace the map $(\phi,Y)$ with a map $(\phi, h)$ for a suitable $h$ as found by the following lemma. Finiteness of $\phi$ will be handled later using a Noether normalization argument.

\begin{lemma}\label{hexists} Let $c\in \N$ be as in Lemma \ref{phiexists}. Let $d>>0$ be an integer such that for every $i>(d-c)/{3}$, 
$$ \hash (\A^1_F)_{\scr =i} > dm. $$ 
Let  $\phi\in A_{\scr \leq d}$  be as given by  \eqref{phiexists} and $S:=\syn{Z}(\phi)\intersection Z$ . Then, there exists $h \in F[X,Y]$ such that
\begin{enumerate}
 \item $h_{|S} :S \rightarrow \A^1$ is radicial, i.e. injective and preserves the degree.
 \item The map $\A^2_F \xrightarrow{(\phi,h)} \A^2_F$ is \'{e}tale at all $x\in S$. 
 \item $h_{|\syn{Z}(f)}: \syn{Z}(f) \to \A^1_F$ is a finite map. 
\end{enumerate}
\end{lemma}
\begin{proof} \underline{Step(1)}: In this step we will show that with the given choice of $d$, it is possible to choose $h_1$ which satisfies condition (1) of the Lemma. \\
 We claim that 
$$ \hash S_{=i} \leq \hash (\A^1_F)_{\scr =i} \ \ \forall \ i\geq 1.$$ 
As explained in Remark \ref{sfinite}, $\syn{Z}(\phi)\intersection \syn{Z}(f)$ is finite. By Bezout theorem, $\hash S\leq \dg(\phi)\dg(f) = dm$. Thus the above claim holds for all $i> (d-c)/{3}$ by the choice of $d$. On the other hand, the claim also holds for $i\leq (d-c)/{3}$, since by Lemma \ref{phiexists}, $Y$ is radicial at $S_{\scr \leq (d-c)/{3}}$. Thus we can choose a set theoretic map $S \xrightarrow{\tilde{h}} \A^1_F$ which is injective and preserves degree of points. By Chinese remainder theorem, there exists an $h_1\in F[X,Y]$ such that for all $x\in S$
$$ h_1(x)= \tilde{h}(x).$$\\

\noindent \underline{Step(2)}: Now, using the $h_1$ from above step, we will find a $h_2\in F[X,Y]$ which satisfies conditions (1) and (2) of the Lemma. 
It is sufficient to find an $h_2\in F[X,Y]$ such that 
\begin{align*}
(i) \ \  h_2  \equiv h_1  &\ \ \ \ \syn{mod} \ {\mathfrak m}_x \ \ \forall \ x\in S \\
(ii) \ \ \frac{\partial h_2}{\partial X}(x) & = 0  \ \ \ \forall\  x\in S \\
(iii) \ \ \frac{\partial h_2}{\partial Y}(x) & = 1 \ \ \ \forall \ x\in S
\end{align*} 

First, we claim that for any closed point $x\in \A^2_F$,  there exists an $h_x \in F[X,Y]$ such that 
\begin{align*}
 \ \  h_x  \equiv h_1  &\ \ \ \ \syn{mod} \ {\mathfrak m}_x  \\
 \ \ \frac{\partial h_x}{\partial X}(x) & = 0   \\
 \ \ \frac{\partial h_x}{\partial Y}(x) & = 1 
\end{align*}
We choose a polynomial $f_1\in F[X]$ such that $f_1(x)=0$ and $\partial f_1/\partial X(x)\neq 0$. To see that such a choice is possible, let $\pi_1:\A^2_F\to \A^1_F$ be the projection on to the $X$-coordinate. The minimal polynomial of any primitive element of the residue field of $\pi_1(x)$ satisfies our requirement. Similarly, we choose $f_2\in F[Y]$ such that $f_2(x)=0$ and $\partial f_2/\partial Y(x)\neq 0.$ Using Chinese remainder theorem and the fact that the residue field $F(x)$ is perfect, we choose $g_1,g_2 \in F[X,Y]$ such that 
$$g_1^p(x) = - \frac{ \partial h_1/ \partial X(x)}{\partial f_1/ \partial X(x)},$$
$$g_2^p(x) =  \frac{ \big(1-\partial h_1/ \partial Y(x)\big)}{\partial f_2/ \partial Y(x)}.$$
We leave it to the reader that 
$$ h_x = h_1 + g_1^p f_1 + g_2^p f_2 $$
satisfies the requirement of our claim. 
Now, by Chinese remainder theorem, there exists $h_2\in F[X,Y]$ such that 
$$ h_2 \equiv h_x \ \ \ \syn{mod} \ {\mathfrak m}_x^2 \ \ \forall \ x\in S.$$
It is straightforward to see that $h_2$ satisfies conditions (1) and (2) of the Lemma. \\

\noindent \underline{Step (3)}: Choose a non-constant polynomial $\beta \in F[Y]$ such that 
$\beta(x)=0$ for all $x\in S$. Since $f$ is monic in $X$,  $\syn{Z}(f)\xrightarrow{Y} \A^1_F$ is a finite map. Thus $\beta :\syn{Z}(f) \to \A^1_F$ is also a finite map. As $\syn{dim}(\syn{Z}(f))=1$, for a sufficiently large integer $\ell$, $$h:=h_2-\beta^{p\ell}$$
defines a finite map $\syn{Z}(f) \xrightarrow{h} \A^1_F$ by Noether normalization trick (see \eqref{nntrick1}).   Clearly $h$ continues to satisfy conditions (1) and (2) of the Lemma since $\beta^{p\ell} \in {\mathfrak m}_x^2$ for all $x\in S$. 

\end{proof}

\begin{proof}[Proof of theorem \ref{dtwo}]
 Let $\phi,h$ be as in Lemmas \ref{phiexists} and \ref{hexists} respectively. Let  $\widetilde{\Phi}$ be the map $\A^2_F \xrightarrow{(\phi,h)} \A^{2}_F$, and $\widetilde{\Psi}:=\phi$. Recall that $\tilde{S}:= \phi^{-1}(0)\intersection \syn{Z}(f)$ (with reduced scheme structure). By Remark \ref{sfinite} it is finite. 

\noindent \underline{Step 1}:  We claim that there exists a $g\in F[X,Y]$ such that if $W_g:=\A^2_F\backslash \syn{Z}(g)$, then  $\widetilde{\Phi}(\tilde{S})\subset W_g$ and  
$$ \widetilde{\Phi}_{|\widetilde{\Phi}^{-1}(W_g)\intersection \syn{Z}(f)} : \widetilde{\Phi}^{-1}(W_g)\intersection \syn{Z}(f) \longrightarrow W_g$$ 
is a closed immersion. The proof of this claim is a repetition of the argument in \cite[3.5.1]{chk} (see also \eqref{nakarg}). Let $\{p, x_1,...,x_n\}$ be the set of points in $\tilde{S}$. Since $\widetilde{\Phi}$ is \'{e}tale and radicial at all points of $\tilde{S}$ (see \eqref{phiexists}(2) and \eqref{hexists}(1) ) we have 
$ \widetilde{\Phi}^{-1}(\widetilde{\Phi}(\tilde{S})) \to \A^2_F$ is a closed immersion.  Let $y_0,...,y_n$ be the points in $\widetilde{\Phi}(\tilde{S})$. Let $\eta_i$ be the maximal ideal in $F[X,Y]$ corresponding to the closed point $y_i$. Thus the above closed immersion gives us a surjective map 
$$ F[X,Y] \twoheadrightarrow  \frac{ F[\syn{Z}(f)]}{ \eta_0\cdots \eta_n}$$
where $F[\syn{Z}(f)]$ is the coordinate ring of $\syn{Z}(f)$. 
If $C$ denotes the cokernel of $F[X,Y] \to F[\syn{Z}(f)]$ (as $F[X,Y]$ modules), then the above surjective map implies that $$C \tensor \frac{F[X,Y]}{\eta_0\cdots \eta_n} = 0 .$$
Note that $\widetilde{\Phi}_{| \syn{Z}(f)}$ is a finite map, since $h$ is a finite map (\eqref{hexists}(3)). Thus $F[\syn{Z}(f)]$ is a finite $F[X,Y]$ module.  Thus, by Nakayama's lemma, there exists an element $g \in F[X,Y]$ such that 
$g \notin \eta_0\cdots \eta_n $ and $C_g=0$. In other words, the map 
$$ F[X,Y]_g \twoheadrightarrow  F[\syn{Z}(f)]_g$$ 
is surjective. This proves the claim since if $W_g:= \A^2_F \backslash \syn{Z}(g)$, the above surjectivity is equivalent to the following being a closed immersion 
$$ \widetilde{\Phi} : \syn{Z}(f) \intersection \widetilde{\Phi}^{-1}(W_g)  \to W_g.$$ \\

\noindent \underline{Step 2}: Let $E$  be the smallest closed subset of $\syn{Z}(f)$ satisfying the following three conditions 
\begin{enumerate}
\item[(i)] $x\in E$ if $x\in \syn{Z}(f)$ and $\widetilde{\Phi}$ is not \'{e}tale at $x$.
\item[(ii)]  $\syn{Z}(f)\backslash Z\subset E$. 
\item[(iii)] $\syn{Z}(f) \backslash \big(\widetilde{\Phi}^{-1}(W_g) \intersection \syn{Z}(f)\big) \subset E$. 
\end{enumerate}
Since $\tilde{S}$ contains at least one point in each irreducible component of $\syn{Z}(f)$, (iii) implies that $E$ is finite (see also Remark \ref{sfinite}). Moreover, by the above step and condition (iii) we have 
$$ \syn{Z}(f) \backslash E \longrightarrow \A^2_F$$ 
is a locally closed immersion. Moreover, $\tilde{S}$ and $E$ are disjoint, and hence $\phi(p) \notin \phi(E)$. Since $E$ is finite, we choose a non-constant polynomial expression $Q$ in $h$ which vanishes on $p$ as well as $E$. For an integer $\ell >>0$ and divisible by $\Char(F)$, we claim that $(\phi-Q^\ell, h)$ presents $(W,\syn{Z}(f),p)$. Let 
$$ \Phi:= (\phi-Q^\ell, h) \ \ \ \ \text{and} \ \ \ \Psi:= \phi-Q^{\ell}.$$ 

To prove the claim we need to verify the conditions of the Definition (1)-(4) \ref{defpresents}. 
Condition (1), i.e. finiteness of $\Psi_{|\syn{Z}(f)}$, follows by \eqref{nntrick1} since 
 $\ell$ is large, and $h_{|\syn{Z}(f)}$ is finite. As $Q$  vanishes on $p$ and $E$, $\Psi(p)\notin \Psi(E)$ follows from $\phi(p)\notin \phi(E)$. Thus if $S:= \Psi^{-1}\Psi(p)\intersection Z$, then $S \subset \syn{Z}(f)\backslash E$. Conditions (2) to (4) of \eqref{defpresents} follow from the conditions (i) to (iii) of $E$ in the beginning of this step. 
 \end{proof}


\vspace{1.5cm}
\begin{center}
Amit Hogadi, IISER Pune, Dr. Homi Bhabha Road, Pashan, Pune : 411008, INDIA\\
email: amit@iiserpune.ac.in\\
\vspace{0.3cm}
Girish Kulkarni, IISER Pune, Dr. Homi Bhabha Road, Pashan, Pune : 411008, INDIA\\
email: girish.kulkarni@students.iiserpune.ac.in
\end{center}
\end{document}